\documentclass[final]{siamltex}

\usepackage{setspace}
\usepackage{color}
\usepackage[active]{srcltx}
\usepackage{amsmath}
\usepackage{amssymb}
\usepackage{pict2e}
\usepackage{graphicx}
\usepackage{multirow}
\usepackage{epsfig}
\usepackage{verbatim}
\usepackage{subfig}
\usepackage{algorithm}
\usepackage{algpseudocode}

\DeclareMathOperator*{\Min}{minimize}

\newcommand{\qed}{\nobreak \ifvmode \relax \else
      \ifdim\lastskip<1.5em \hskip-\lastskip
      \hskip1.5em plus0em minus0.5em \fi \nobreak
      \vrule height0.75em width0.5em depth0.25em\fi}

\title{Restoration of Images Corrupted by Impulse Noise and Mixed Gaussian Impulse Noise using Blind Inpainting}

\author{Ming~Yan
\thanks{Department of Mathematics, University of California, Los Angeles, CA, 90095 USA. E-mail: yanm@math.ucla.edu.}
}
\begin{document}

\maketitle

\begin{abstract}
This article studies the problem of image restoration of observed images corrupted by impulse noise and mixed Gaussian impulse noise. Since the pixels damaged by impulse noise contain no information about the true image, how to find this set correctly is a very important problem. We propose two methods based on blind inpainting and $\ell_0$ minimization that can simultaneously find the damaged pixels and restore the image. By iteratively restoring the image and updating the set of damaged pixels, these methods have better performance than other methods, as shown in the experiments. In addition, we provide convergence analysis for these methods, these algorithms will converge to  coordinatewise minimum points. In addition, they will converge to local minimum points (or with probability one) with some modifications in the algorithms.
\end{abstract}

\begin{keywords}
impulse noise, mixed Gaussian impulse noise, total variation, blind inpainting, image restoration, $\ell_0$ minimization
\end{keywords}


\pagestyle{myheadings}
\thispagestyle{plain}
\markboth{MING YAN}{Restoration of Images Corrupted by Impulse Noise using Blind Inpainting}

\section{Introduction}
\label{sec:BIintro}

Observed images are often corrupted by impulse noise during image acquisition and transmission, caused by malfunctioning pixels in camera sensors, faulty memory locations in hardware, or bit errors in transmission~\cite{Bovi05}. There are two common types of impulse noise: salt-and-pepper impulse noise and random-valued impulse noise. Assume that the dynamic range of an image is $[d_{\mbox{\tiny min}},d_{\mbox{\tiny max}}]$. For images corrupted by salt-and-pepper impulse noise, the noisy pixels can take only two values $d_{\mbox{\tiny min}}$ and $d_{\mbox{\tiny max}}$, while for images corrupted by random-valued impulse noise, the noisy pixels can take any random value between $d_{\mbox{\tiny min}}$ and $d_{\mbox{\tiny max}}$.

In this work, the original unknown $M\times N$ image $u$ is defined on a domain {$\Omega=\{(i,j):\ i=1,\dots,M,\ j=1,\dots,N\}$}, and the observed $M\times N$ image $f$ is modeled as
\begin{align}\label{for:problem_all}
f_{i,j}=\left\{\begin{array}{ll} (Hu)_{i,j}+(n_1)_{i,j}, & {(i,j)\in\Omega_1},\\
                                          (n_2)_{i,j}, & {(i,j)\in \Omega_1^c:= \Omega\backslash\Omega_1}.\end{array}\right.
\end{align}
Here, $n_2$ is the impulse noise, and $n_1$ is the additive zero-mean Gaussian white noise. $H$ is the identity or a blurring operator, which is assumed to be continuous. The subset $\Omega_1^c$ of $\Omega$ denotes the region where the information of $Hu$ is missing. The problem is to find the true image $u$ from observed image $f$ given the operator $H$. 

If $\Omega_1^c$ is empty, there is no impulse noise, then we have $f=Hu+n_1$, which is an image denoising (and deblurring) problem, and it has been extensively studied by both signal processing researchers and mathematicians. If $\Omega_1^c$ is not empty and known, this can be considered as an image inpainting (and deblurring) problem.

Here, we will consider the last and most difficult case where $\Omega_1^c$ is not empty and unknown. The challenge of this problem is to restore the lost details, and remove the impulse noise simultaneously. If $n_1=0$, this problem is an impulse noise removal (and deblurring) problem and if $n_1\neq0$ it becomes a mixed Gaussian impulse noise removal (and deblurring) problem. There are already several types of approaches for solving these problems.

The first type of approaches treats $n_2$ as outliers and uses the $\ell_1$ norm in the fidelity term to increase the robustness of inpainting to outliers~\cite{Niko04,BaSK05,BaKS06,Gilo08}, and the problem is to solve
\begin{align}
 \Min_u \sum_{i,j}|(Hu)_{i,j}-f_{i,j}|+\lambda_1J(u),
\end{align}
where $J(u)$ is a regularization on the true image $u$. There are many candidates for the regularization $J(u)$, and some examples are Tikhonov regularization \cite{TikA77}, Geman and Reynolds' half quadratic variational models \cite{GemR92}, Rudin, Osher and Fatemi's total variation models \cite{RuOF92,RudO94}, and framelet based models \cite{CaCS08,LiSD11}. This approach does not need to find the damaged pixels and performs well in impulse noise removal. However, for the case of images corrupted by mixed Gaussian impulse noise, the Gaussian noise is not treated properly.

The second type of approaches is the two-stage approach \cite{ChHN04,ChHN05,CaCN08, CaCN10, LiSD11,Xiao2011,rodriguez-2012-mixed}, which estimates the inpainting region $\Omega_1^c$ before estimating $u$. In these approaches, the second stage becomes a regular image inpainting (and deblurring) problem \cite{BeSC00,BeVS03,BeVS03b,ChSZ06}
\begin{align}
 \Min_u {1\over2}\sum_{(i,j)\in \Omega_1}((Hu)_{i,j}-f_{i,j})^2+\lambda_1J(u).
\end{align}
The success of these two-stage approaches relies on the accurate detection of $\Omega_1^c$, e.g. adaptive median filter (AMF) \cite{GonW01} is used to detect salt-and-pepper impulse noise, while adaptive center-weighted median filter (ACWMF) \cite{ACWMF} and rank-ordered logarithmic difference (ROLD) \cite{DoCX07} are utilized to detect random-valued impulse noise.

Though adaptive median filter can detect most pixels damaged by salt-and-pepper impulse noise, it is more difficult to detect pixels corrupted by random-valued impulse noise than salt-and-pepper impulse noise. Recently, by considering two different types of noise, Dong et al. \cite{BlindInpainingDong} proposed a new method using framelet to remove random-valued impulse noise plus Gaussian noise by solving
\begin{align}
 \Min_{u,v} {1\over2}\sum_{i,j}((Hu)_{i,j}+v_{i,j}-f_{i,j})^2+\lambda_1 \|Wu\|_1 + \lambda_2\sum_{i,j}|v_{i,j}|,
\end{align}
where $W$ is a transformation from the image to the framelet coefficients. Two unknowns $u$ (restored image) and $v$ (noise) are introduced into this variational model, and their methods can simultaneously find $u$ and $v$ using split Bregman iterations \cite{GolO09}.

Dong et al.'s method uses $\ell_1$ norm as a convex approximation of $\ell_0$ term to make the result $v$ sparse, and keep the problem convex in the meantime. However, using non-convex optimization ($\ell_p$ when $p<1$) has better performance than convex optimization in dealing with the sparsity, as shown in compressive sensing~\cite{nonconvex1}. Even $\ell_0$ minimization and smoothed $\ell_0$ minimization are used in many algorithms~\cite{SL0,PenaltyL0,Blumensath08,ZhangDL11,DongZ12,PenaltyL02}. In this paper, we will use $\ell_0$ minimization instead of $\ell_1$ minimization in the problem, and by using $\ell_0$ minimization, the problem of finding $u$ can be solved by considering a problem of  finding $u$ and $\Omega_1$. In addition, using alternating minimization algorithm, it can be solved easily by alternately solving the image inpainting problem and finding the damaged pixels.

The work is organized as follows. In section~\ref{sec:BlindInpaintng} and~\ref{sec:BIAOP}, we introduce our general methods for removing impulse noise using two different treatments for the $\ell_0$ term: I) the $\ell_0$ term is put in the objective function, II) the $\ell_0$ term is in the constraint. The algorithms for these two models are similar. The convergence analysis of these two algorithms is shown in section~\ref{sec:BIconv}. These algorithms will converge to  coordinatewise minimum points. In addition, they will converge to local minimum points (or with probability one) with some modifications in the algorithms. Some experiments are given in section~ \ref{sec:Experiments} to show the efficiency of the proposed methods for removing impulse noise and mixed Gaussian impulse noise. We will end this work by a short conclusion section.

\section{Blind Inpainting Models using $\ell_0$ Term}
\label{sec:BlindInpaintng}
\subsection{Formulation}
\label{sec:formulation}

For an $M\times N$ image, $\Lambda\in\{0,1\}^{M\times N}$ is a binary matrix representing a subset $\Omega_s$ of the pixels as follows:
\begin{align}
 \Lambda_{i,j} = \left\{\begin{array}{cl}
                        1, &\mbox{ if pixel } (i,j) \in \Omega_s,\\
			0, &\mbox{ otherwise}.
                       \end{array}
\right.
\end{align}
The connection between binary matrix $\Lambda$ and subset $\Omega_s$ will be used many times in the follow.

Given a degraded image $f$, our objective is to estimate the damaged (or missing) pixels and restore them. We propose the following model using $\ell_0$ minimization to solve this problem:
\begin{align}\label{for:problem_BI}
 \Min\limits_{u,v}F^P(u,v)\equiv{1\over2}\sum_{i,j}((Hu)_{i,j}+v_{i,j}-f_{i,j})^2+\lambda_1 J(u) + \lambda_2\|v\|_0,
\end{align}
where $J(u)$ is the regularization term on the image, $\lambda_1$ and $\lambda_2$ are two positive parameters. Here $P$ means that $\ell_0$ term is used as a penalty term in the objective function. The parameter $\lambda_1$ is dependent on the noise level of $n_1$. The higher the noise level, the larger the parameter should be. The parameter $\lambda_2$ is dependent on the noise level of impulse noise. The difference from Dong et al.'s method is that $\ell_1$ norm is replaced by $\ell_0$ term. It is difficult to solve this problem because of the $\ell_0$ term in the function. $\ell_0$ term makes the problem non-convex and the objective function is non-continuous. Because what we need to find is just $u$, we can eliminate $v$ from problem~(\ref{for:problem_BI}) by defining $E^P_0(u)$ as $\min\limits_vF^P(u,v)$,  and the problem becomes 
\begin{align}\label{for:problem_BI_uonly}
\Min_u E^P_0(u)=\min_vF^P(u,v).
\end{align}
However, $E^P_0(u)$ is still non-convex and difficult to solve. In order to solve this problem, we will transform the problem into a continuous and multi-convex problem of $u$ and $\Lambda$ by introducing a new variable $\Lambda$, and by solving the new problem of $u$ and $\Lambda$, we can obtain a local optimal solution for the original problem~(\ref{for:problem_BI_uonly}) of $u$ only.

First of all, we provide the intuition behind choosing $\ell_0$ minimization instead of $\ell_1$ minimization for $v$. Because of the speciality of $\ell_0$ minimizations, $v$ can be easily eliminated from $\min\limits_vF^P(u,v)$ and we can obtain the function of $u$ only as follows: 
\begin{align}\label{for:problem_BI_only}
E^P_0(u)= {1\over2}\sum_{i,j} R_0((Hu)_{i,j}-f_{i,j}) + \lambda_1 J(u) ,
\end{align}
where $R_0(x) = \min(|x|^2, 2\lambda_2)$. Similarly, we can obtain the function of $u$ only when $\ell_1$ term is used instead of $\ell_0$ term as follows:
\begin{align}\label{for:problem_BI_only1}
E^P_1(u)\equiv {1\over2}\sum_{i,j}R_1((Hu)_{i,j}-f_{i,j}) + \lambda_1 J(u),
\end{align}
where $R_1(x) = \left\{ \begin{array}{ll}  |x|^2, & \mbox{ if } |x|\leq \lambda_2,\\ 2\lambda_2 |x|-\lambda_2^2, &\mbox{ otherwise}. \end{array}\right.$

The data fidelity terms ${1\over2}\sum_{i,j}R_0((Hu)_{i,j}-f_{i,j})$ and ${1\over2}\sum_{i,j}R_1((Hu)_{i,j}-f_{i,j})$ in problems~(\ref{for:problem_BI_only}) and~(\ref{for:problem_BI_only1}) are used to approximate the negative log-likelihood resulting from the mixed Gaussian impulse noise model, with each $R_0$ or $R_1$ describing the negative log-likelihood for each pixel, because the noise is independently distributed at all pixels. What we are trying to do is finding better and simpler model for mixed Gaussian impulse noise. We can simulate the probability distribution of the pixel values when it is corrupted by both additive Gaussian and random-valued impulse noise. For a fixed pixel value (128 in Fig.~\ref{fig:noise}), a value is added as a Gaussian noise, and it is replaced by any random value between $d_{\mbox{min}}$ and $d_{\mbox{max}}$ (0 and 255 in Fig.~\ref{fig:noise}) with some probability related to the impulse noise level. This is run for $10^8$ times and the approximated negative log-likelihood function is shown in Fig.~\ref{fig:noise}.  We can see that it is a constant when the value is far from the true pixel value. In this figure we also show $R_0(x-128)$ and $R_1(x-128)$ with some scaling and lifting.

\begin{figure*}[!h]
\begin{center}
{\includegraphics[width=0.5\linewidth]{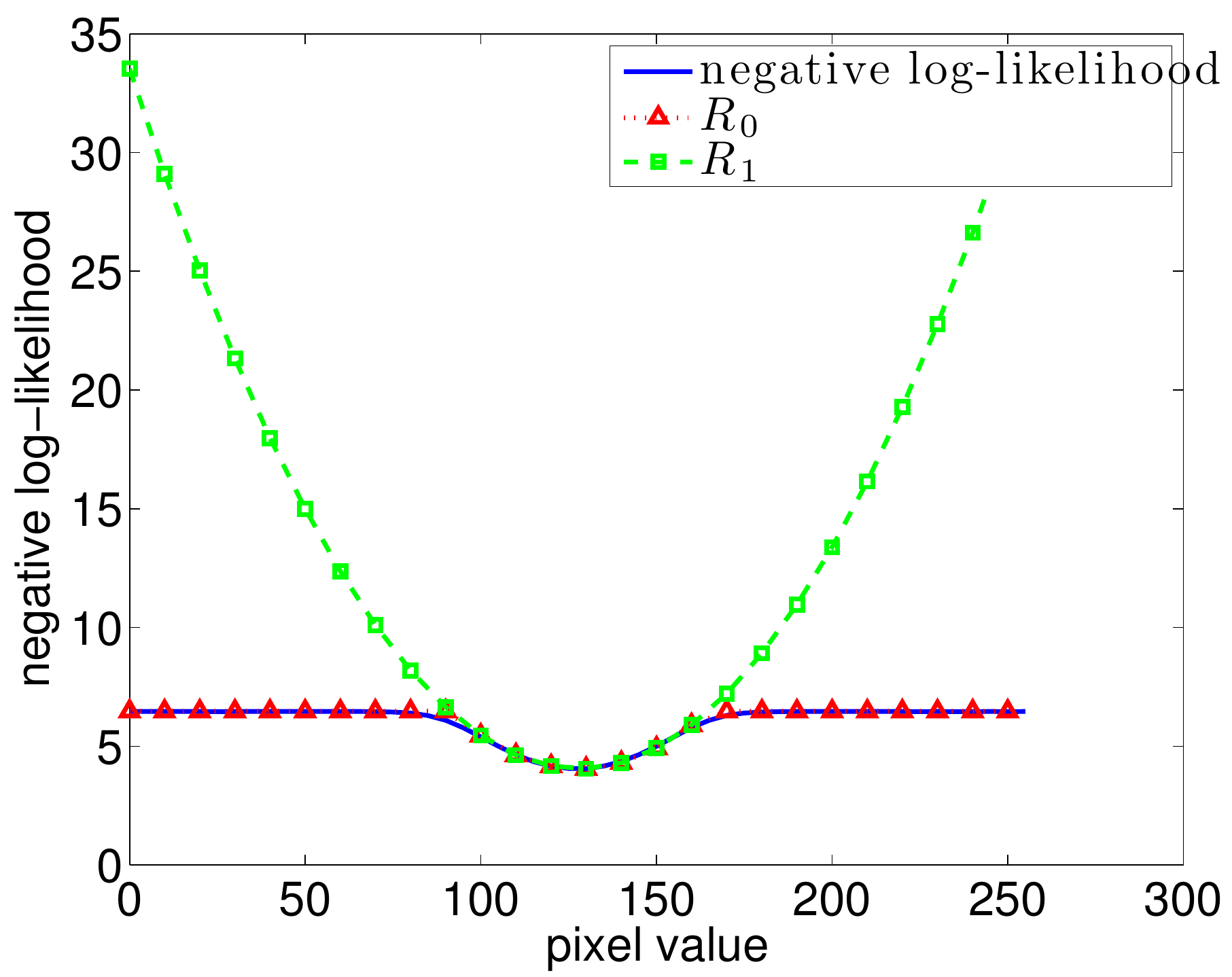}}
\caption{Negative log-likelihood value for the mixed Gaussian impulse noise, and comparison with $R_0$ and $R_1$}
\label{fig:noise}
\end{center}
\end{figure*}

From Fig.~\ref{fig:noise}, we can find that it is more reasonable to use $\ell_0$ term instead of $\ell_1$ term. However, $\ell_1$ minimization makes the problem convex and it is easier to find the solution, while the problem is non-convex and difficult to solve if $\ell_0$ minimization is used. Next, we will introduce an auxiliary variable $\Lambda$ or $\Omega_1$ and we can find a local minimizer of $E^P_0(u)$ by solving the new problem of $u$ and $\Lambda$.

For any fixed $\bar{u}$, we can find the optimal $v$ by solving the following optimization problem: 
$$\Min_v {1\over 2} \sum_{i,j}(H\bar{u}+v-f)^2+\lambda_2\|v\|_0.$$
The solution is 
\begin{align*}
v_{i,j} = \left\{\begin{array}{ll}0, &\mbox{ if }|f_{i,j}-(H\bar{u})_{i,j}|^2< 2\lambda_2,\\
f_{i,j}-(H\bar{u})_{i,j},& \mbox{ if }|f_{i,j}-(H\bar{u})_{i,j}|^2>2\lambda_2,\\
{0 \mbox{ or } f_{i,j}-(H\bar{u})_{i,j}},& \mbox{ if }|f_{i,j}-(H\bar{u})_{i,j}|^2=2\lambda_2.\end{array}\right.
\end{align*}
When $v_{i,j}\neq0$, we have $v_{i,j} = f_{i,j}-(H\bar{u})_{i,j}$. Therefore if we denote
\begin{align}
 \Lambda_{i,j} = \left\{\begin{array}{cl}
			0, &\mbox{ if } v_{i,j} \neq  0,\\
                        1, &\mbox{ if } v_{i,j} = 0,
                       \end{array}
\right.
\end{align}
then we have a new problem with $u$ and $\Lambda$ as follows:
\begin{align}\label{for:problem_BI_2a}
 \Min\limits_{u,\Lambda\in\{0,1\}^{M\times N}}F_1(u,\Lambda)\equiv {1\over2} \sum_{i,j}\Lambda_{i,j}((Hu)_{i,j}-f_{i,j})^2+\lambda_1 J(u)+\lambda_2 \sum_{i,j}(1-\Lambda_{i,j}).
\end{align}
Problem (\ref{for:problem_BI_2a}) can be solved easily by alternating minimization method, and the algorithm for solving (\ref{for:problem_BI_2a}) is described in section~\ref{sec:algorithm}. For a general alternative minimization procedure for convex and non-convex problems, please see~\cite{PTseng}.

{\bf Remark: }In fact, the constraint of $\Lambda\in\{0,1\}^{M\times N}$ can further be relaxed into $\Lambda\in[0,1]^{M\times N}$, and we have the following multi-convex problem:
\begin{align}\label{for:problem_BI_2}
 \Min\limits_{u,\Lambda\in[0,1]^{M\times N}}F_1(u,\Lambda)={1\over2} \sum_{i,j}\Lambda_{i,j}((Hu)_{i,j}-f_{i,j})^2+\lambda_1 J(u)+\lambda_2 \sum_{i,j}(1-\Lambda_{i,j}).
\end{align}

If $u$ is fixed, $F_1(u,\Lambda)$ is a function of $\Lambda$ only and it is separable. The optimal solution for $\Lambda$ with fixed $u$ is 
\begin{align*}
 \Lambda_{i,j} = \left\{\begin{array}{cl}
			0, &\mbox{ if } |f_{i,j}-(Hu)_{i,j}|^2> 2\lambda_2,\\
                        1, &\mbox{ if } |f_{i,j}-(Hu)_{i,j}|^2< 2\lambda_2,\\
                        t, & \mbox{ if } |f_{i,j}-(Hu)_{i,j}|^2= 2\lambda_2,
                       \end{array}
\right.
\end{align*}
where $t$ is 0 or 1 for unrelaxed problem (\ref{for:problem_BI_2a}), and $t$ is any number in $[0,1]$ for relaxed problem (\ref{for:problem_BI_2}). If we eliminate $\Lambda$ as before, $\min_{\Lambda\in\{0,1\}^{M\times N}}F_1(u,\Lambda)$ and $\min_{\Lambda\in[0,1]^{M\times N}}F_1(u,\Lambda)$ are functions with respect to $u$ only and same as $E^P_0(u)$. Because all the problems are non-convex, there may exist many local optimal solutions which may not be global optimal solutions. We will show in section~\ref{Equivalency} that by solving problem (\ref{for:problem_BI_2}), we can obtain a local minimizer of function $E^P_0(u)$.

\subsection{Algorithm}
\label{sec:algorithm}
The objective function defined in (\ref{for:problem_BI_2}) is non-convex. It is still difficult to solve it in the pair $(u,\Lambda)$, but we can use alternating minimization method, which separates the energy minimization over $u$ and $\Lambda$ into two steps. For solving the problem in $u$ with $\Lambda$ fixed, it is a convex optimization problem for image inpainting and the problem of finding $\Lambda$ with $u$ fixed can be solved in one step. These two subproblems are

1) Finding $u$: Given an estimate of the support matrix $\Lambda$, the minimization over $u$ is just an image inpainting (and deblurring) problem \cite{ChanShen2005}:
\begin{align}\label{for:problem_Inpainting}
 \Min\limits_{u}{1\over2}\sum_{(i,j)\in\Omega_1}((Hu)_{i,j}-f_{i,j})^2+\lambda_1 J(u).
\end{align}
There are many existing methods for solving this problem.

2) Finding $\Lambda$: Given an estimate of the image $u$, the minimization over $\Lambda$ becomes:
\begin{align}
 \Min\limits_{\Lambda\in[0,1]^{M\times N}}{1\over2} \sum_{i,j}\Lambda_{i,j}((Hu)_{i,j}-f_{i,j})^2-\lambda_2 \sum_{i,j}\Lambda_{i,j}.
\end{align}
Since this minimization problem of $\Lambda$ is separable, it can be solved exactly in only one step:
\begin{align}\label{for:problem_UpdateLambda}
 \Lambda_{i,j}=\left\{\begin{array}{cc}
                      0&\mbox{ if }\left((Hu)_{i,j}-f_{i,j}\right)^2/2> \lambda_2,\\
                      0\mbox{ or }1&\mbox{ if }\left((Hu)_{i,j}-f_{i,j}\right)^2/2=\lambda_2,\\
		      1&\mbox{ if }\left((Hu)_{i,j}-f_{i,j}\right)^2/2 < \lambda_2.\\
                     \end{array}
\right.
\end{align}
Therefore, the proposed algorithm for blind inpainting with $\ell_0$ minimization is iteratively finding $u$ and $\Lambda$. As mentioned in section~\ref{sec:formulation}, $\min_{\Lambda\in\{0,1\}^{M\times N}}F_1(u,\Lambda)=\min_{\Lambda\in\{0,1\}^{M\times N}}F_1(u,\Lambda)$ for all fixed $u$, thus we can force $\Lambda_{i,j}\in\{0,1\}$ during the algorithm. When $\left((Hu)_{i,j}-f_{i,j}\right)^2/2=\lambda_2$, we can randomly choose $\Lambda_{i,j}$ to be 0 or 1.

The detailed algorithm for blind inpainting is described below, the initial $\Lambda^0$ is chosen by the methods for detecting the impulse noise (AMF for salt-and-pepper impulse noise and ACWMF for random-valued impulse noise). Usually three iterations are sufficient, as shown in the experiments.

\begin{algorithm}[H]
\caption{Proposed blind inpainting algorithm.}\label{alg:TV_BI}
\begin{algorithmic}
	\State \textbf{Input:} $f$, $\lambda_1$, $\lambda_2$, $\Lambda^0$, $\epsilon$
	\State \textbf{Initialization:} $k=1$.
	\While{$k< 2$ or $F_1(u^k,\Lambda^k)-F_1(u^{k-1},\Lambda^{k-1})> \epsilon$}								\vspace{0.1cm}
                	\State \mbox{Obtain }$u^k$ \mbox{ by solving (\ref{for:problem_Inpainting})}.
                	\State \mbox{Obtain }$\Lambda^k$ \mbox{ by } (\ref{for:problem_UpdateLambda}).
		\State $k=k+1$.
	\EndWhile
\end{algorithmic}
\end{algorithm}

Here $\epsilon$ is chosen to be small and served as a stopping criteria to stop the algorithm when the difference in function values between two iterations is too small. $\lambda_1$ and $\lambda_2$ are two parameters depending on the noise levels of $n_1$ and $n_2$.

{\bf Remark: }This algorithm and the algorithm in next section depend on the initial $\Lambda^0$, and choosing a better $\Lambda^0$ will reduce the total number of iterations and the restoration result (because of the non-convexity of the problem). Therefore, we can choose the result of AMF and ACWMF for salt-and-pepper and random-valued impulse noise respectively as initial $\Lambda^0$. For salt-and-pepper impulse noise, AMF will provide a very accurate initial guess for $\Lambda$ for most cases and the improvement from more iterations is not too much, while for random-valued impulse noise, the output of ACWMF is not accurate, and more iterations are needed to improve the detection of corrupted pixels and the recovery result.

\section{Blind Inpainting Using Adaptive Outlier Pursuit}
\label{sec:BIAOP}

In the previous section, we proposed a method for blind inpainting by putting a $\ell_0$ term in the objective function, which can be solved by iteratively updating the set of pixels damaged by impulse noise (or the binary matrix $\Lambda\in\{0,1\}^{M\times N}$) and restoring the image. Instead of putting the $\ell_0$ term in the objective function, we can also put a constraint on $\|v\|_0$, which will be equivalent to a constraint on $\sum_{i,j}\Lambda_{i,j}$. This technique has been applied to robust 1-bit compressive sensing where there are sign flips in the binary measurements belonging to $\{-1,1\}$~\cite{YaYO12a} and robust matrix completion~\cite{YaYO12b}. We proposed an algorithm, named adaptive outlier pursuit (AOP), which can adaptively find the sign flips (outliers) and reconstruct the signal by using other measurements assumed to be correct. Since images corrupted by impulse noise can also be considered as sparsely corrupted measurements, the same idea can be applied in impulse noise (and mixed Gaussian impulse noise) removal by iteratively finding the pixels corrupted by impulse noise and recovering the image using other pixels.

Let us assume that the number of pixels corrupted by impulse noise is bounded above by a integer $L$, this can be obtained from the noise level of impulse noise. Therefore, the new problem is
\begin{align}\label{for:problem_AOP}\left.\begin{array}{rl}
 \Min\limits_{u,v}\ &{1\over2}\sum_{i,j}((Hu)_{i,j}+v_{i,j}-f_{i,j})^2+\lambda_1 J(u),\\
 \mbox{subject to }& \|v\|_0\leq L,\end{array}\right.
\end{align}
which can be written as 
\begin{align}\label{pro:F^C}
\Min\limits_{u,v} F^C(u,v)\equiv{1\over2}\sum_{i,j}((Hu)_{i,j}+v_{i,j}-f_{i,j})^2+\lambda_1 J(u)+{\iota}_{\{v:\|v\|_0\leq L\}},
\end{align}
where $\iota_{\{v:\|v\|_0\leq L\}}$ is the indicator function equals to zero when $\|v\|_0\leq L$ and $+\infty$ otherwise. Here $C$ means that the $\ell_0$ term is put in the constraint. We can further eliminate $v$, and the problem of $u$ only is:
\begin{align}
\Min_u E_0^C(u)\equiv \min_vF^C(u,v).
\end{align}

Similarly, we can introduce new variation $\Lambda$ and the corresponding problem of $u$ and $\Lambda$ is
\begin{align}\label{for:problem_AOP2}\left.\begin{array}{rl}
 \Min\limits_{u,\Lambda}& \sum_{i,j}{1\over2}\Lambda_{i,j}((Hu)_{i,j}-f_{i,j})^2+\lambda_1 J(u),\\
 \mbox{subject to }& \sum_{i,j}(1-\Lambda_{i,j})\leq L,\quad \Lambda_{i,j}\in\{0,1\},\end{array}\right.
\end{align}
which can be described in another way as
\begin{align}\label{pro:F_2}
\Min_{u,\Lambda\in\{0,1\}^{M\times N}}F_2(u,\Lambda)\equiv \sum_{i,j}{1\over2}\Lambda_{i,j}((Hu)_{i,j}-f_{i,j})^2+\lambda_1 J(u)+\iota_{\{\Lambda:\sum_{i,j}(1-\Lambda_{i,j})\leq L\}},
\end{align}
where $\iota_{\{\Lambda:\sum_{i,j}(1-\Lambda_{i,j})\leq L\}}$ is the indicator function equals to zero when $\sum_{i,j}(1-\Lambda_{i,j})\leq L$ and $+\infty$ otherwise.
Note that we can also relax the constraint $\Lambda\in\{0,1\}^{M\times N}$ to $\Lambda\in[0,1]^{M\times N}$ and the problem becomes a multi-convex problem, as done in the previous section. In addition we have $\min_{\Lambda\in\{0,1\}^{M\times N}}F_2(u,\Lambda)=\min_{\Lambda\in[0,1]^{M\times N}}F_2(u,\Lambda)=E_0^C(u)$.

This problem can also be solved iteratively as in the previous section. The $u$-subproblem is same as the previous one, and the $\Lambda$-subproblem is slightly different. In order to update $\Lambda$, we have to solve
\begin{align}\label{for:problem_AOP_Lambda}\left.\begin{array}{rl}
 \Min\limits_{\Lambda}&\sum_{i,j}\Lambda_{i,j}((Hu)_{i,j}-f_{i,j})^2/2,\\
 \mbox{subject to}&\sum_{i,j}(1-\Lambda_{i,j})\leq L, \Lambda_{i,j}\in\{0,1\}.
 \end{array}
 \right.
\end{align}
This problem is to choose $M\times N-L$ elements with least sum from $M\times N$ elements $\{( (H u)_{i,j}-f_{i,j})^2/2\}_{i=1,j=1}^{M,N}$. Given a $u$ estimated from (\ref{for:problem_Inpainting}), we can update $\Lambda$ in one step:
\begin{align}
\label{for:problem_UpdateLambda2}
\Lambda_{i,j} =
\left\{
 \begin{array}{ll}
 0,\ &\text{if}\ ((Hu)_{i,j}-f_{i,j})^2/2\geq\lambda_2,\\
 1,\ &\text{if}\ ((Hu)_{i,j}-f_{i,j})^2/2<\lambda_2,
 \end{array}
 \right.
\end{align}
where $\lambda_2$ is the $L^{th}$ largest term of $\{((Hu)_{i,j}-f_{i,j})^2/2\}_{i=1,j=1}^{M,N}$. If the $L^{th}$ and $(L+1)^{th}$ largest terms are equal, then we can choose any binary matrix $\Lambda$ such that $\sum_{i,j}\Lambda_{i,j}=M\times N-L$ and
\begin{align}
\min_{i,j,\Lambda_{i,j}=0}((Hu)_{i,j}-f_{i,j})^2/2\geq \max_{i,j,\Lambda_{i,j}=1}((Hu)_{i,j}-f_{i,j})^2/2.
\end{align}
The algorithm for blind inpainting using AOP is described below.

\begin{algorithm}[H]
\caption{Proposed blind inpainting using AOP.}\label{alg:AOP}
\begin{algorithmic}
	\State \textbf{Input:} $f$, $\lambda_1$, $L$, $\Lambda^0$, $\epsilon$
	\State \textbf{Initialization:} $k=1$.
	\While{$k< 2$ or $F_2(u^k,\Lambda^k)-F_2(u^{k-1},\Lambda^{k-1})> \epsilon$}
	      	\vspace{0.1cm}
                   \State \mbox{Obtain }$u^k$ \mbox{ by solving (\ref{for:problem_Inpainting})}.
                   \State \mbox{Obtain }$\Lambda^k$ \mbox{ by } (\ref{for:problem_UpdateLambda2}).
		\State $k=k+1$.
	\EndWhile
\end{algorithmic}
\end{algorithm}
Here $\epsilon$, $\Lambda^0$ and $\lambda_1$ are chosen in the same way as algorithm~\ref{alg:TV_BI}. The integer $L$ is an estimation of the number of corrupted pixels, which is easy to obtain from the noise level of the impulse noise.

The difference between these two algorithms is the $\Lambda$-subproblem, the threshold $\lambda_2$ is fixed for algorithm~\ref{alg:TV_BI}, while $\lambda_2$ is changing for AOP. AOP can be considered as one special case of algorithm~\ref{alg:TV_BI} with changing $\lambda_2$. However the performance of these two algorithms is similar, and the parameter $L$ is easier to obtain than $\lambda_2$. So we will only use algorithm~\ref{alg:AOP} for numerical experiments.

\section{Convergence Analysis}
\label{sec:BIconv}
\label{Equivalency}

In this section, we establish some convergence results for these two algorithms. We will show that these two algorithms will stop in finite steps, and the output is a coordinatewise minimum point of $F_1(u,\Lambda)$ (or $F_2(u,\Lambda)$) with relaxed constraint $\Lambda_{i,j}\in[0,1]$. A point $(\tilde{u},\tilde{\Lambda})$ is a coordinatewise minimum point of $F_1(u,\Lambda)$ (or $F_2(u,\Lambda)$) means that $\tilde{u}$ is a minimizer of $F_1(u,\tilde{\Lambda})$ (or $F_2(u,\tilde{\Lambda})$) and $\tilde{\Lambda}$ is a minimizer of $F_1(\tilde{u},\Lambda)$ (or $F_2(\tilde{u},\Lambda)$). In addition, we can modify a little bit in the algorithm and the output will be a local minimizer of $E_0^P(u)$ (or $E_0^C(u)$) (or with probability one).

Since the convergence analysis is similar for both algorithms, let $F(u,\Lambda)$ stand for $F_1(u,\Lambda)$ in the penalty problem~(\ref{for:problem_BI_2a}) and $F_2(u,\Lambda)$ in the constraint problem~(\ref{pro:F_2}), and $E_0(u)$ stand for $E_0^P(u)$ and $E_0^C(u)$ as $\min_{\Lambda\in\{0,1\}^{M\times N}}F(u,\Lambda)$.

Before deriving the convergence results of the algorithms, two theorems are introduced to show that we can solve problems~(\ref{for:problem_BI_2a}) and~(\ref{pro:F_2}) of $u$ and $\Lambda$ to find a local minimizer for $E_0(u)$.

\begin{theorem}\label{lemma1}
If $u^*$ is a local minimizer of $E_0(u)$, then for any $\Lambda^*\in \{0,1\}^{M\times N}$ minimizing $F(u^*,\Lambda)$, $(u^*,\Lambda^*)$ is a local minimizer of $F(u,\Lambda)$. 
\end{theorem}
\begin{proof}
Since $u^*$ is a local minimizer of $E_0(u)$, we can find $\epsilon>0$ such that for all $u$ satisfying $\|u-u^*\|<\epsilon$, we have $E_0(u)\geq E_0(u^*)$. Therefore for all $(u,\Lambda)$ satisfying $\|(u,\Lambda)-(u^*,\Lambda^*)\|<\epsilon$, we have $\|u-u^*\|<\epsilon$. 
\begin{align*}
F(u,\Lambda) \geq E_0(u) \geq E_0(u^*) = F(u^*,\Lambda^*).
\end{align*}
Thus $(u^*,\Lambda^*)$ is a local minimizer of $F(u,\Lambda)$.
\end{proof}

To obtain a sufficient condition for $u^*$ being a local minimizer of $E_0(u)$, we have the following theorem.
\begin{theorem}\label{sufficient_minimum}
Given fixed $u^*$, if for all $\bar{\Lambda}\in \{0,1\}^{M\times N}$ minimizing $F(u^*,\Lambda)$, we also have that $u^*$ minimizing $F(u,\bar{\Lambda})$, then $u^*$ is a local minimizer of $E_0(u)$.
\end{theorem}
\begin{proof}We will prove it for the penalty problem~(\ref{for:problem_BI_2a}) with $F(u,\Lambda)=F_1(u,\Lambda)$ first. Let ${\Omega_+}= \{(i,j): |f_{i,j}-(Hu^*)_{i,j}|^2> 2\lambda_2\}$ and ${\Omega_-}= \{(i,j): |f_{i,j}-(Hu^*)_{i,j}|^2< 2\lambda_2\}$, then from the continuity of $H$, we can find $\epsilon>0$ such that when $\|u-u^*\|<\epsilon$, we have $|f_{i,j}-(Hu)_{i,j}|^2> 2\lambda_2$ for all $(i,j)\in{\Omega_+}$ and $|f_{i,j}-(Hu)_{i,j}|^2< 2\lambda_2$ for all $(i,j)\in{\Omega_-}$. Notice that $E_0(u)=\min_{\Lambda\in\{0,1\}^{M\times N}}F(u,\Lambda)$, then there exists $\bar{\Lambda}\in\{0,1\}^{M\times N}$ such that $E_0(u)=F(u,\bar{\Lambda})$. We have $\bar{\Lambda}_{i,j}=0$ when $(i,j)\in\Omega_+$ and $\bar{\Lambda}_{i,j}=1$ when $(i,j)\in\Omega_-$. Thus $\bar{\Lambda}$ is also a minimizer of $F_1(u^*,\Lambda)$. In addition, we have $E_0(u^*)=F(u^*,\bar{\Lambda})\leq F(u,\bar{\Lambda}) = F(u)$. Therefore, $u^*$ is a local minimizer of $E_0(u)$.

For the constraint problem~(\ref{pro:F_2}) with $F(u,\Lambda)=F_2(u,\Lambda)$, we have to just replace $\lambda_2$ with the $L^{th}$ largest term of $\{((Hu)_{i,j}-f_{i,j})^2\}_{i=1,j=1}^{M,N}$, and the result follows.
\end{proof}

{\bf Remark:} We can replace the $\{0,1\}^{M\times N}$ with the relaxed version $[0,1]^{M\times N}$ and both theorems are still valid.

With these theorems, we are ready to shown the convergence results of the two algorithms.

\begin{theorem}
Both algorithms will converge in finite steps and the output $(u^*, \Lambda^*)$ is a coordinatewise minimum point of $F(u,\Lambda)$.
\end{theorem}
\begin{proof}
As explained before, though we have $\Lambda_{i,j}\in[0,1]$ for the relaxed problem, we can always force $\Lambda_{i,j}\in\{0,1\}$ during the iterations because for every fixed $\bar{u}$, we can also find a minimizer of $F(\bar{u},\Lambda)$ satisfying $\Lambda_{i,j}\in\{0,1\}$. Since $\Lambda_{i,j}\in\{0,1\}$, there are only finite number of $\Lambda$'s (the total number of different $\Lambda$'s with constraint $\Lambda_{i,j}\in\{0,1\}$ is $2^{M\times N}$) and the algorithm will stop in finite steps if the $u$-subproblem and $\Lambda$-subproblem are solved exactly. Assume that at step $k$, the function $F(u,\Lambda)$ stops decreasing, which means
\begin{align}
F(u^k,\Lambda^k) = F(u^{k+1},\Lambda^{k+1}).
\end{align}
Together with nonincreasing property of the algorithm
\begin{align}
F(u^k,\Lambda^k) \geq F(u^{k+1},\Lambda^k) \geq F(u^{k+1},\Lambda^{k+1}),
\end{align}
we have
\begin{align}
F(u^k,\Lambda^k) = F(u^{k+1},\Lambda^k) = F(u^{k+1},\Lambda^{k+1}).
\end{align}
Thus
\begin{align}
F(u^k,\Lambda^k)&=F(u^{k+1},\Lambda^k)=\min_{u}F(u,\Lambda^k),\\
F(u^k,\Lambda^k)&=\min_{\Lambda}F(u^k,\Lambda).
\end{align}
Then $(u^*,\Lambda^*)=(u^k,\Lambda^k)$ is a coordinatewise minimum point of $F(u,\Lambda)$.
\end{proof}

However, coordinatewise minimum point $(\tilde{u},\tilde{\Lambda})$ may not be a local minimum point of $F(u,\Lambda)$. As shown in the next theorem, $\tilde{\Lambda}$ being the unique minimum point of $F(\tilde{u},\Lambda)$ is a sufficient condition for $(\tilde{u},\tilde{\Lambda}$) to be a local minimum point.

\begin{theorem}\label{localminimumsuff}For a coordinatewise minimum point $(u^*,\Lambda^*)$ of $F(u,\Lambda)$, if $\Lambda^*$ is the unique minimum point of $F(u^*,\Lambda)$, then $(u^*,\Lambda^*)$ is a local minimum point of $F(u,\Lambda)$. Furthermore, $u^*$ is a local minimum point of $E_0(u)$.
\end{theorem}
\begin{proof}
{  Since $\Lambda^*$ is the unique minimum point of $F(u^*,\Lambda)$ and $u^*$ minimizes $F(u,\Lambda^*)$, we have $u^*$ being a local minimum point of $E_0(u)$ from theorem~\ref{sufficient_minimum}. Then $(u^*,\Lambda^*)$ being a local minimum point of $F(u,\Lambda)$ follows from theorem~\ref{lemma1}.}
\end{proof}

Let $(u^*,\Lambda^*)$ be a coordinatewise minimum point of $F(u,\Lambda)$. From theorem~\ref{localminimumsuff}, if $u^*$ is not a local minimum point of $E_0(u)$, then there are many minimum points for $F(u^*,\Lambda)$. { In addition, from theorem~\ref{sufficient_minimum}, there exists another minimum point $\bar{\Lambda}$ of $F(u^*,\Lambda)$, such that $F(u^*,\bar{\Lambda})>\min_{u} F(u,\bar{\Lambda})$.}

Based on theorems~\ref{sufficient_minimum} and~\ref{localminimumsuff}, we have the following two corollaries.

\begin{corollary}
When solving the subproblem of finding $\Lambda^k$, if there are many minimum points for $F(u^k,\Lambda)$, we can choose the best $\Lambda$ with lowest $\min_uF(u,\Lambda)$, then the algorithm will stop at a local minimum point of $E_0(u)$. 
\end{corollary}
\begin{proof}When the algorithm stops at $(u^*,\Lambda^*)$, we have $F(u^*,\bar{\Lambda})=\min_uF(u,\bar{\Lambda})$ for all $\bar{\Lambda}\in \{0,1\}^{M\times N}$ minimizing $F(u^*,\Lambda)$, otherwise, we can find a $\Lambda$ giving lower $\min_uF(u,\Lambda)$ and the algorithm continues. Then from theorem~\ref{sufficient_minimum}, we know that $u^*$ is a local minimum point of $E_0(u)$.
\end{proof}

{\bf Remark: }This strategy for choosing $\Lambda$ may not be very useful for practical computation because there could be a large number of candidates for such $\Lambda$. However, in the numerical experiments, this does not happen, because of the rounding error in the calculation and the probability for two values to equal is 0.

Instead of having to choose the best candidates from many candidates, we can modify the function $F(u,\Lambda)$ to avoid this case, as Wang et al. did in~\cite{WangY}. 
\begin{corollary}
If we can modify the objective function $F(u,\Lambda)$ by adding $\tau\sum_{i,j}\Lambda_{i,j}r_{i,j}$, where $\{r_{i,j}\}$ are random values uniformly distributed in $[0,1]$ and $\tau$ is a small number, then the algorithm will stop at a local minimum of $\tilde{E}_0(u)\equiv\min_{\Lambda\in\{0,1\}^{M\times N}}(F(u,\Lambda)+\tau\sum_{i,j}\Lambda_{i,j}r_{i,j})$ with probability one.
\end{corollary}
\begin{proof}
In this case, the subproblem for updating $\Lambda$ is 
\begin{align}
 \Lambda_{i,j}=\left\{\begin{array}{cc}
                      0&\mbox{ if }\left((Hu)_{i,j}-f_{i,j}\right)^2/2+\tau r_{i,j}> \lambda_2,\\
                      0\mbox{ or }1&\mbox{ if }\left((Hu)_{i,j}-f_{i,j}\right)^2/2 +\tau r_{i,j}=\lambda_2,\\
		      1&\mbox{ if }\left((Hu)_{i,j}-f_{i,j}\right)^2/2 +\tau r_{i,j}< \lambda_2.\\
                     \end{array}
\right.
\end{align}
The probability of getting $\left((Hu)_{i,j}-f_{i,j}\right)^2/2 +\tau r_{i,j}=\lambda_2$ is 0 because of the randomness of $r_{i,j}$. Then the algorithm will converge to a local minimum of $\tilde{E}_0(u)$ with probability one. Similarly for AOP, the probability of $L^{th}$ and $(L+1)^{th}$ largest term being equal is 0.
\end{proof}

{\bf Remark: }Besides adding additional term onto $F(u,\Lambda)$ we can also add a small random variable $\tau\Lambda_{i,j} r_{i,j}$ onto $f_{i,j}$. 

\section{Numerical Experiments}
\label{sec:Experiments}

Because the performance of these two algorithms is similar and it is easier to obtain an approximate value for the number of damaged pixels, in this section, we apply the blind inpainting algorithm using AOP to remove impulse noise and mixed Gaussian impulse noise. 

As mentioned in the introduction, there are many different choices for $J(x)$, and the performance of this algorithm depends on which $J(x)$ is chosen. In order to make the comparison with other methods fair, we choose $J(x)$ to be total variation for all methods used in the numerical experiments. Therefore the optimization problem will be different from the ones that proposed in the literature. For example, we implement the algorithm of paper~\cite{BlindInpainingDong} by replacing the wavelet frame with total variation~\cite{6100364}.

If the regularization for image $u$ is total variation, and $H$ is the identity operator, the step for finding $u$ is
\begin{align}
\Min\limits_{u}\sum_{i,j}{1\over2}\Lambda_{i,j}(u_{i,j}-f_{i,j})^2+\lambda_1 \mbox{TV}(u),
\end{align}
which is the famous TV inpainting model \cite{ChaS02,ChaS05b}. Numerous algorithms proposed for solving TV denoising problem can be adopted to solve this TV inpainting problem with some necessary modifications. Some examples are algorithms based on duality \cite{Cham04,ZhWC10}, augmented Lagrangian methods \cite{TaiW09,WuTa10}, and split Bregman iterations \cite{GolO09,tvreg}. In the numerical experiments, we choose the split Bregman iteration to solve the subproblem.

To evaluate the quality of the restoration results, peak signal to noise ratio (PSNR) is employed. Given an image $u\in[0,255]^{M\times N}$, the PSNR of the restoration result $\hat{u}$ is defined as follows:
\begin{align}
 \mbox{PSNR}(\hat{u},u)=10\log_{10}{255^2\over {1\over MN}\sum\limits_{i,j}(\hat{u}_{i,j}-u_{i,j})^2}.
\end{align}

There are two important types of impulse noise: salt-and-pepper impulse noise and random-valued impulse noise. The pixels damaged by salt-and-pepper impulse noise are much easier to find since the values are either $d_{\mbox{\tiny min}}$ or $d_{\mbox{\tiny max}}$. The adaptive median filter (AMF) has been widely used to accurately identify most pixels damaged by salt-and-pepper impulse noise (See e.g. \cite{HwaH95,GonW01}). The detection of pixels corrupted by random-valued impulse noise is much harder than salt-and-pepper impulse noise because the value of damaged pixels can be any number between  $d_{\mbox{\tiny min}}$ and $d_{\mbox{\tiny max}}$. ACWMF was proposed to detect pixels damaged by random-valued impulse noise.

For the first experiment, salt-and-pepper impulse noise is considered. Because the pixels corrupted by this kind of impulse noise can only take two values, the detection of damaged pixels is easy. As an efficient method for detecting the damaged pixels, AMF is used widely in salt-and-pepper impulse noise removal. We will compare total variation blind inpainting using AOP with AMF and TVL1, where TVL1 is the result of solving the following problem,
\begin{align}\label{for:problem_TV_l1}
 \Min\limits_{u}\sum_{i,j}|u_{i,j}-f_{i,j}|+\lambda \ \mbox{TV}(u),
\end{align}
using split Bregman~\cite{tvreg}. The parameter $\lambda$ is tuned to achieve the best quality of the restoration images. 

\iffalse
\else
Four test images are corrupted by Gaussian noise of zero mean and standard deviations $\sigma = 5, 10, 15$, then we add salt-and-pepper impulse noise with different levels ($s=30\%,50\%,70\%$) on the test images, with or without the Gaussian noise. The PSNR values of the results from three methods are summarized in Table~\ref{tab:Denoise_IN}.

\begin{table*}[!h]
\footnotesize
\begin{center}
\begin{tabular}{|r|cccc|cccc|}\hline
\multicolumn{9}{|c|}{\bf Salt-and-Pepper Impulse Noise}\\\hline
\multirow{2}{*}{$\sigma+s$}&\multicolumn{4}{|c|}{``Lena''} & \multicolumn{4}{|c|}{``House''} \\\cline{2-9}
    & Noisy & AMF & TVL1 & AOP &Noisy & AMF & TVL1 & AOP \\\hline
    0+30\%& 10.68 & 33.80 & 30.97 & {\bf 37.75}& 10.42 & 38.97 & 36.53 & {\bf 47.14}\\\hline
 5+30\%  & 10.66 & 31.47 & 30.32 & {\bf 34.56} & 10.40 & 33.69 & 34.49 & {\bf 39.09}  \\\hline
 10+30\% & 10.62 & 27.93 & 29.40 & {\bf 32.25} & 10.39 & 28.90 & 32.79 & {\bf 35.73}  \\\hline
 15+30\% & 10.54 & 25.14 & 28.59 & {\bf 30.41} & 10.30 & 25.66 & 31.41 & {\bf 33.49}  \\\hline
    0+50\%& 8.44 & 30.35 & 27.98& {\bf 33.98}& 8.22 & 34.60 & 31.70 & {\bf 42.50}\\\hline
 5+50\%  &  8.45 & 29.00 & 27.58 & {\bf 32.61} & 8.19  & 31.73 & 31.36 & {\bf 37.43}  \\\hline
 10+50\% &  8.42 & 26.54 & 27.25 & {\bf 30.88} & 8.18  & 27.84 & 30.37 & {\bf 34.60}  \\\hline
 15+50\% &  8.40 & 24.15 & 26.46 & {\bf 29.50} & 8.14  & 24.96 & 29.60 & {\bf 32.50}  \\\hline
    0+70\%& 7.00 & 26.85 &24.90 & {\bf 30.61}& 6.75 & 30.05 & 26.80 & {\bf 36.84}\\\hline
 5+70\%  &  6.97 & 26.11 & 24.65 & {\bf 29.94} & 6.73  & 28.80 & 26.81 & {\bf 34.44}  \\\hline
 10+70\% &  6.98 & 24.62 & 24.57 & {\bf 29.05} & 6.74  & 26.15 & 26.36 & {\bf 32.28}  \\\hline
 15+70\% &  6.97 & 22.83 & 24.20 & {\bf 28.11} & 6.72  & 23.86 & 25.85 & {\bf 31.23}  \\\hline
 &\multicolumn{4}{|c|}{``Cameraman''} & \multicolumn{4}{|c|}{``Boat''} \\\cline{2-9}
    & Noisy & AMF & TVL1 & AOP &Noisy & AMF & TVL1 & AOP  \\\hline
     0+30\%& 10.32 & 33.62 & 30.43 & {\bf 38.43}&10.70 & 30.16&27.64 & {\bf 33.32}\\\hline
5+30\%  & 10.28 & 31.34 & 30.09 &{\bf 35.33}  & 10.69 & 30.27 & 27.84 & {\bf 33.39}\\\hline
 10+30\% & 10.25 & 28.15 & 29.35 &{\bf 32.47}  & 10.69 & 30.34 & 27.70 & {\bf 33.06}\\\hline
 15+30\% & 10.21 & 25.39 & 28.40 &{\bf 30.45}  & 10.68 & 30.24 & 27.70 & {\bf 32.96}\\\hline
    0+50\%& 8.08 & 29.78 & 26.80& {\bf 34.58}& 8.49& 27.27& 25.00& {\bf 30.54}\\\hline
 5+50\%  &  8.08 & 28.35 & 26.55 &{\bf 32.78}  & 8.51  & 27.20 & 25.02 & {\bf 30.49}\\\hline
 10+50\% &  8.09 & 26.38 & 26.16 &{\bf 30.57}  & 8.48  & 27.24 & 25.24 & {\bf 30.19}\\\hline
 15+50\% &  8.04 & 24.29 & 26.09 &{\bf 29.08}  & 8.49  & 27.12 & 25.02 & {\bf 30.12}\\\hline
     0+70\%& 6.62 & 25.73 & 23.20& {\bf 29.85}& 7.02&24.33 & 22.42& {\bf 27.20}\\\hline
5+70\%  &  6.62 & 25.22 & 23.21 &{\bf 29.06}  & 7.02  & 24.19 & 22.35 & {\bf 27.14}\\\hline
 10+70\% &  6.60 & 24.02 & 23.01 &{\bf 28.04}  & 7.02  & 24.18 & 22.42 & {\bf 27.08}\\\hline
 15+70\% &  6.60 & 22.48 & 22.66 &{\bf 26.96}  & 7.01  & 24.23 & 22.37 & {\bf 26.97}\\\hline
\end{tabular}
\caption{PSNR(dB) for denoising results of different algorithms for noisy images corrupted by salt-and-pepper
impulse noise and mixed Gaussian impulse noise. $\sigma$ is the standard deviation for the Gaussian noise and $s$ is the level of salt-and-pepper impulse noise.}
\label{tab:Denoise_IN}
\end{center}
\end{table*}

From Table \ref{tab:Denoise_IN}, we can see that for salt-and-pepper impulse noise, the results from total variation blind inpainting using AOP are better than those by AMF and TVL1 for all noise levels. The visual comparison of some results is shown in Fig.~\ref{fig:Denoise_10_30}. We can see noisy artifacts in the background of the images obtained by AMF, and the images obtained by TVL1 are blurred with some lost details. Images restored by total variation blind inpainting using AOP are smooth in flat regions of the background and the details are kept. 

We do not compare AOP with two-stage approaches because the detection of damaged pixels by salt-and-pepper impulse noise using AMF is very accurate, and $\Lambda$ will not change too much during the iterations, thus the performance of AOP will be similar to the two-stage approach by first detecting the damaged pixels by AMF and then solve the total variation image inpainting problem, which is just the first iteration of AOP. Therefore, for the cases with salt-and-pepper impulse noise, using blind inpainting does not improve too much by adaptively updating $\Lambda$, and our focus will be on random-valued impulse noise.

\begin{figure*}[!h]
\begin{center}
{\includegraphics[width=0.24\linewidth]{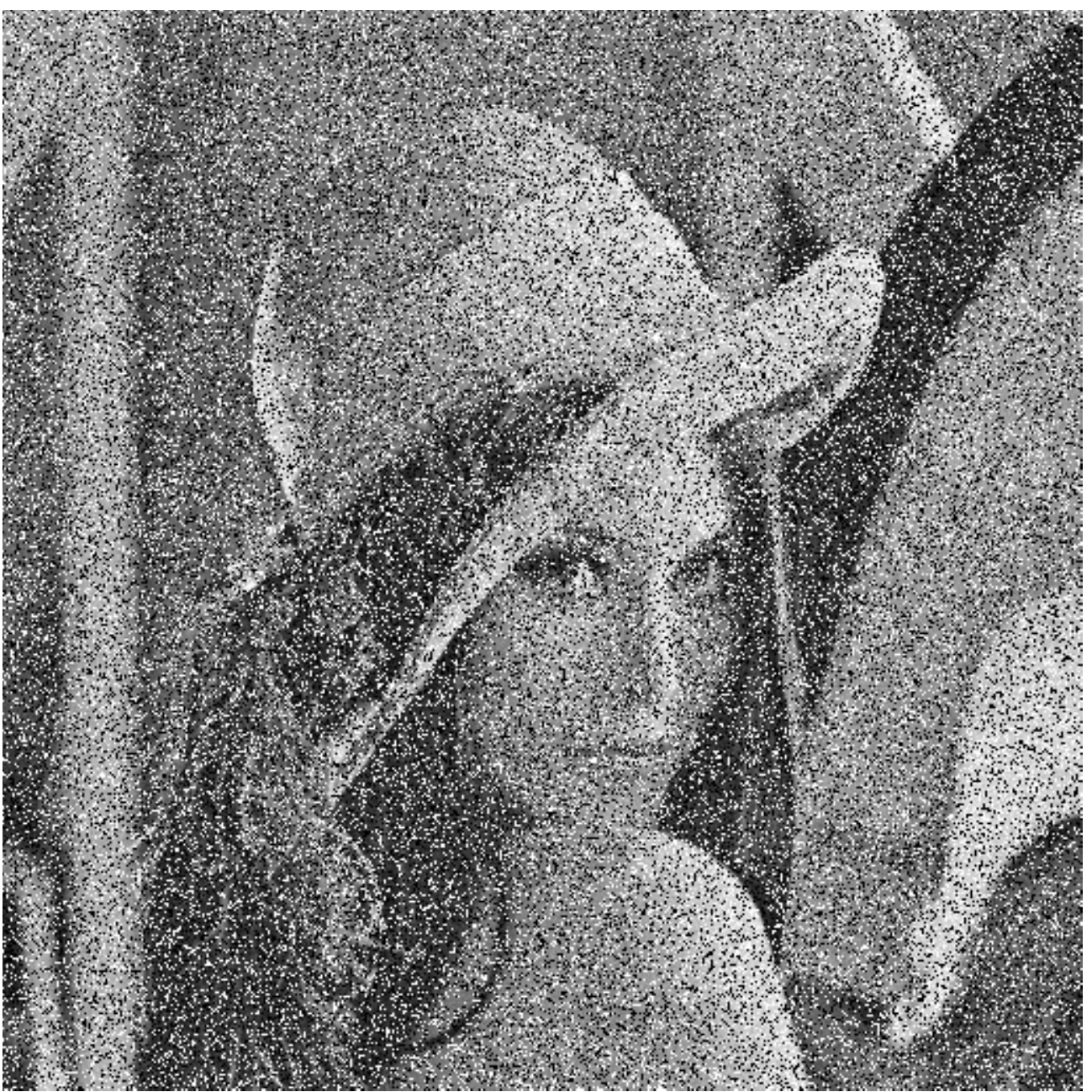}}
{\includegraphics[width=0.24\linewidth]{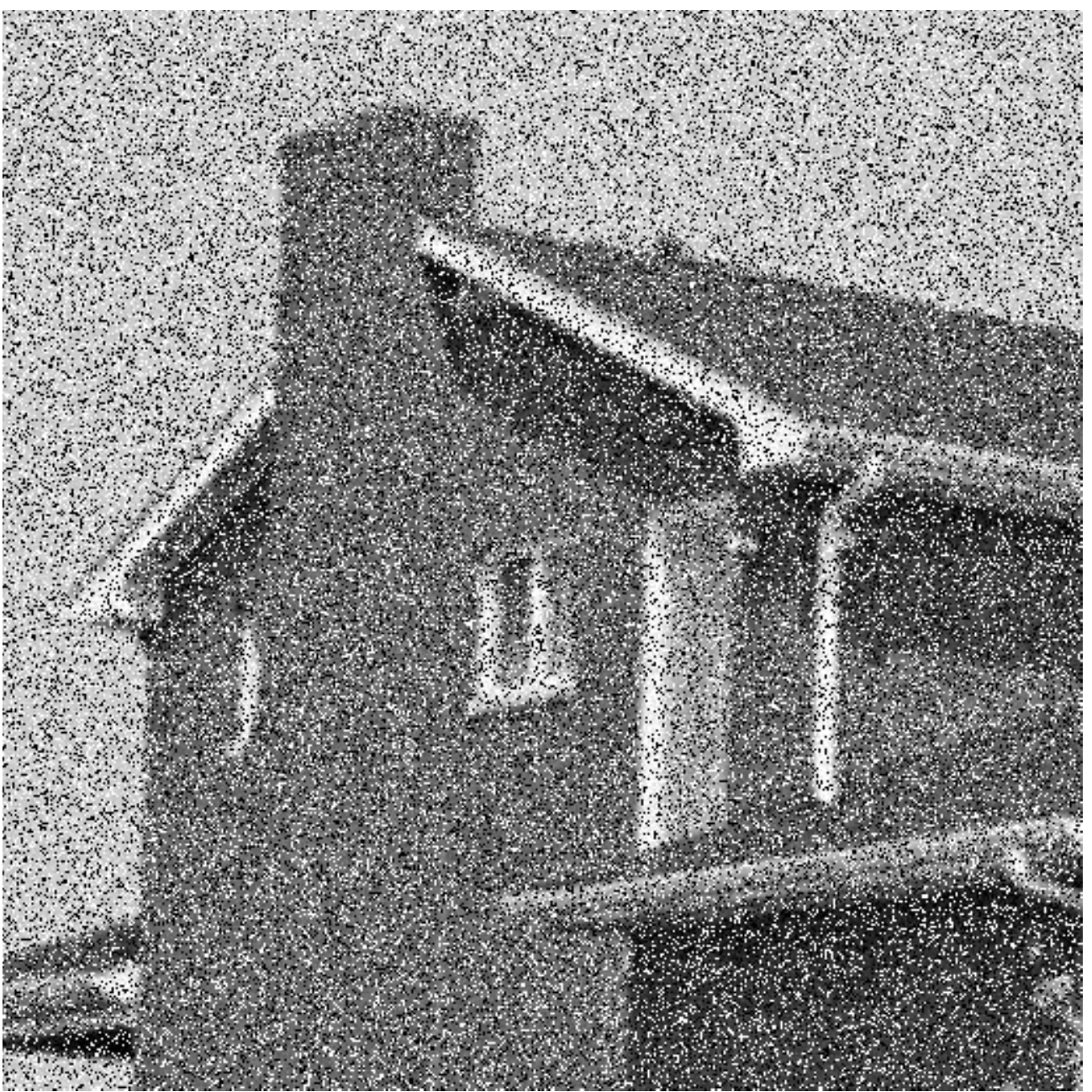}}
{\includegraphics[width=0.24\linewidth]{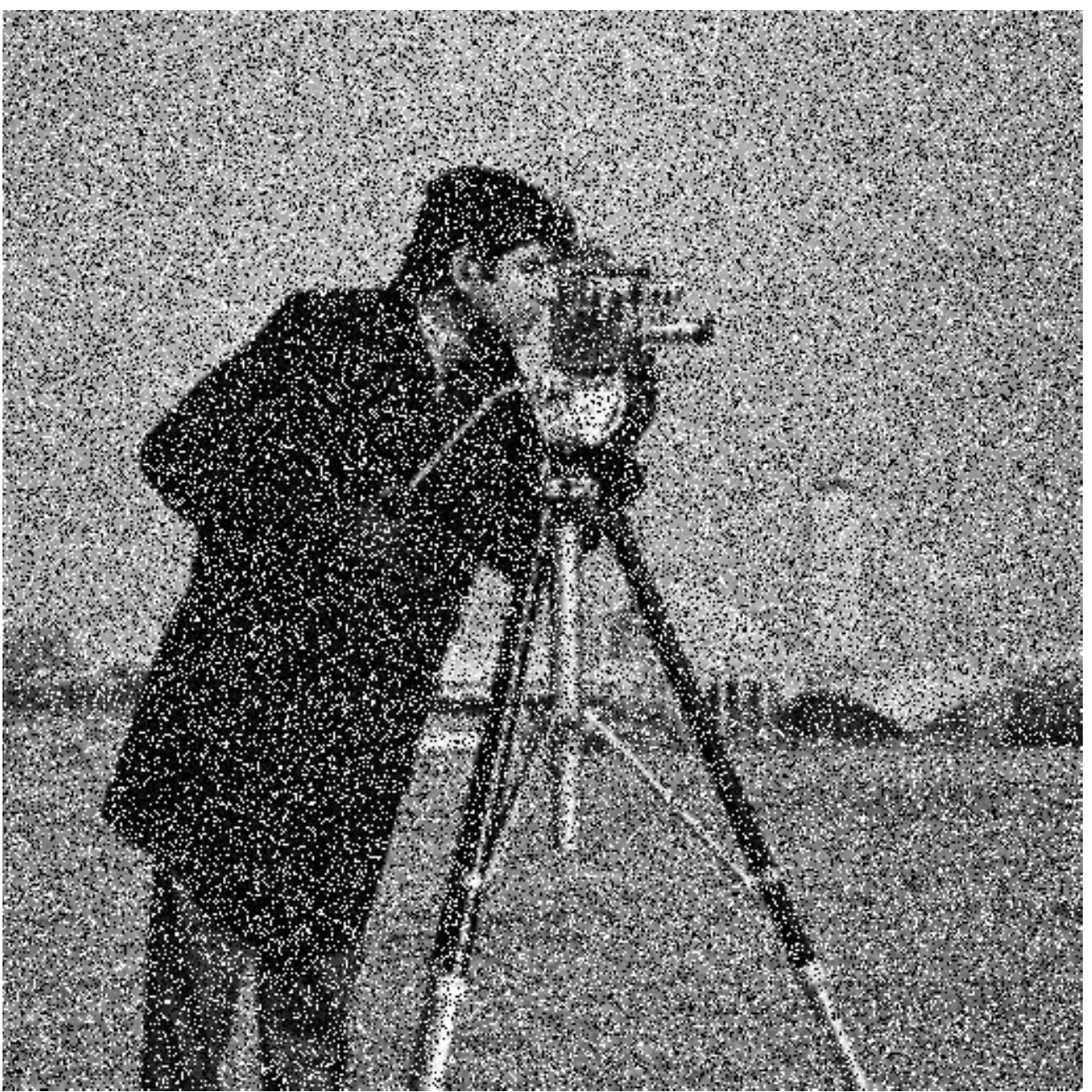}}
{\includegraphics[width=0.24\linewidth]{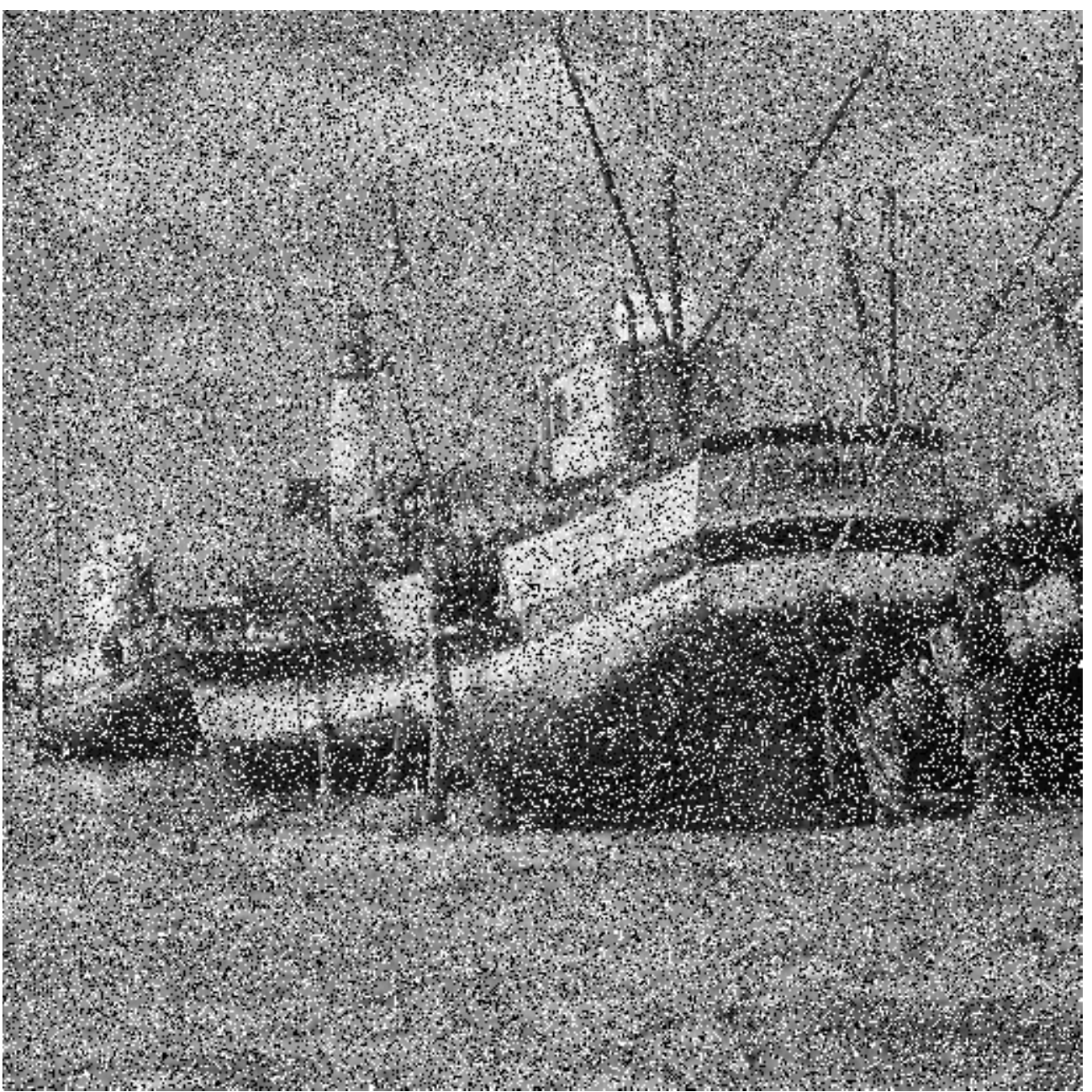}}\\
{\includegraphics[width=0.24\linewidth]{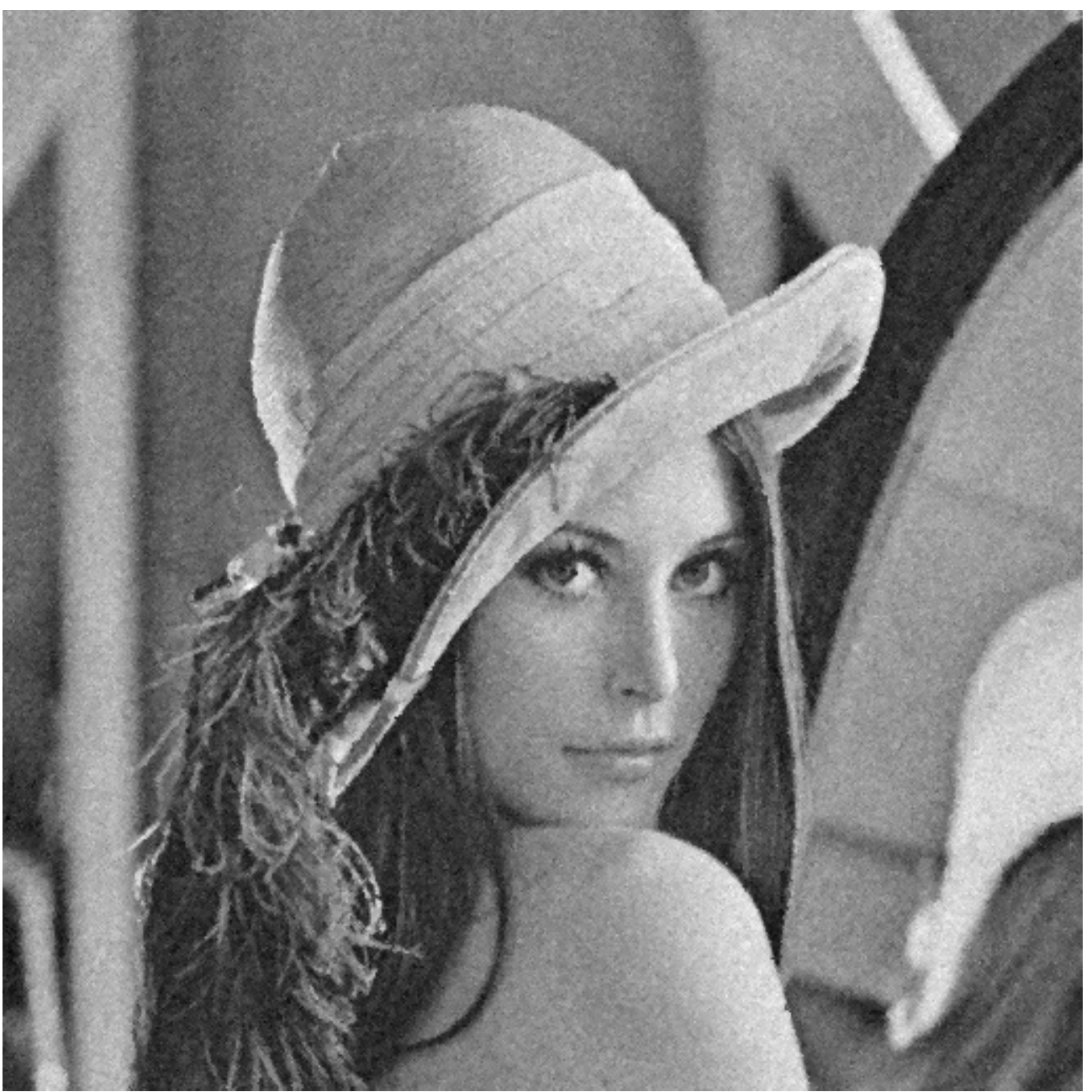}}
{\includegraphics[width=0.24\linewidth]{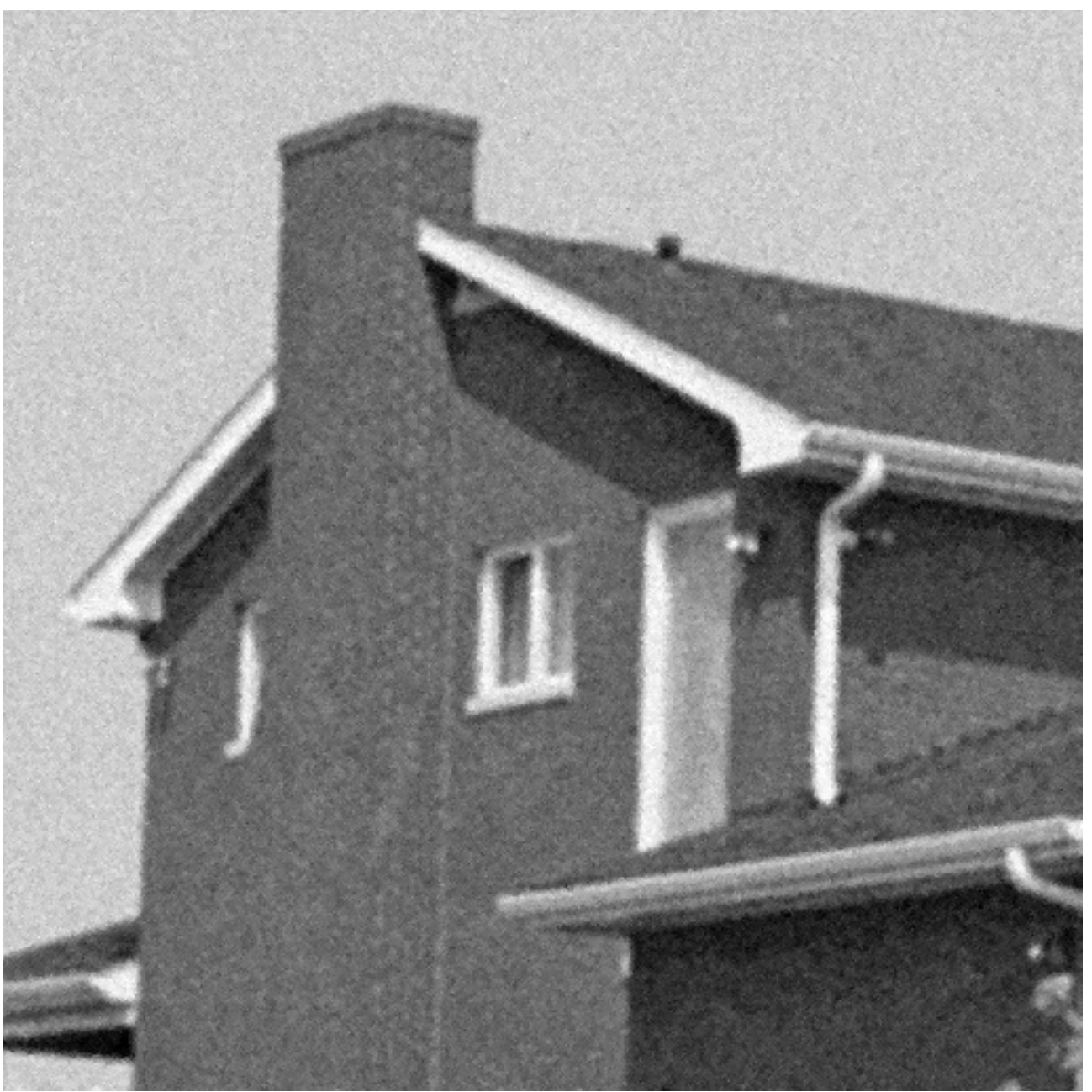}}
{\includegraphics[width=0.24\linewidth]{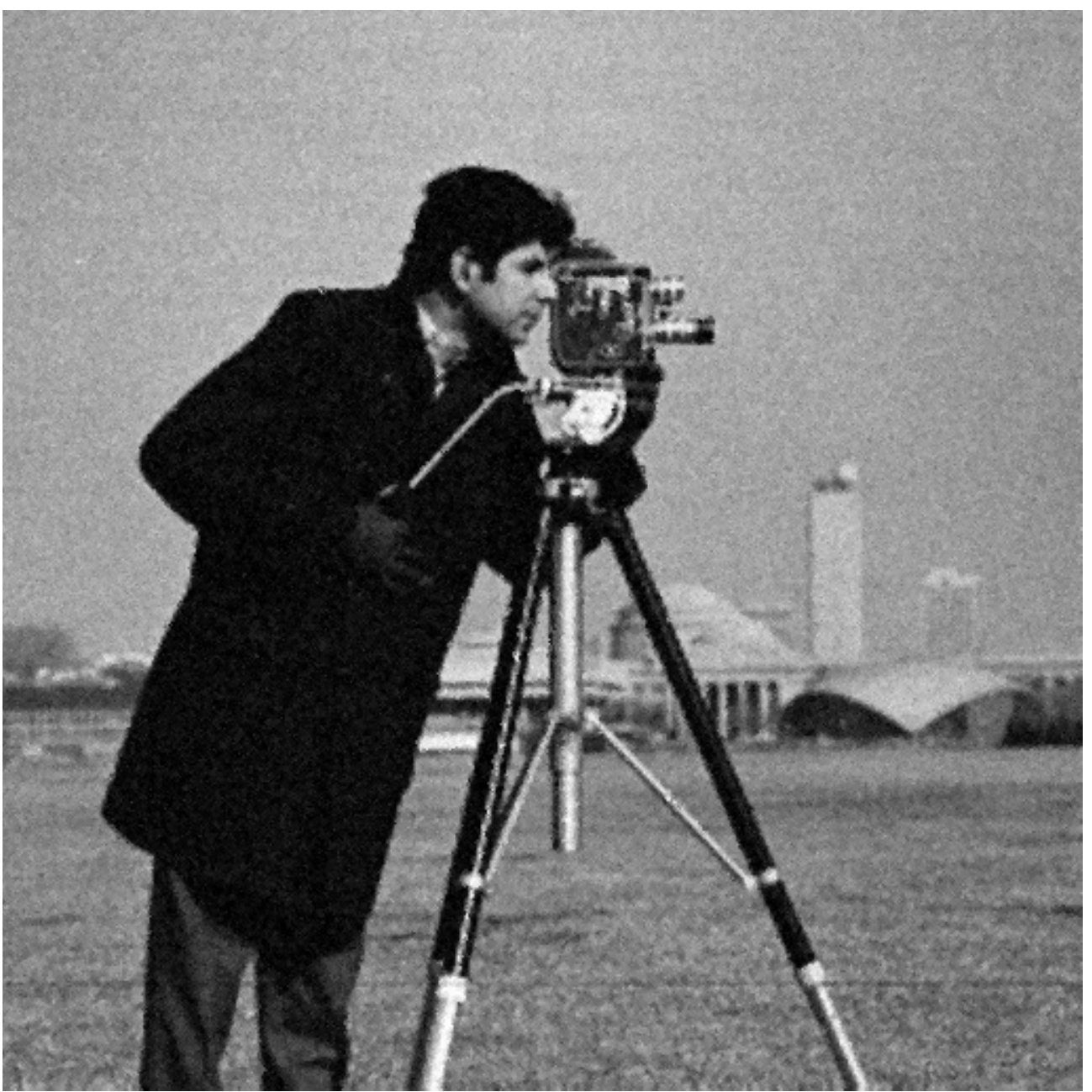}}
{\includegraphics[width=0.24\linewidth]{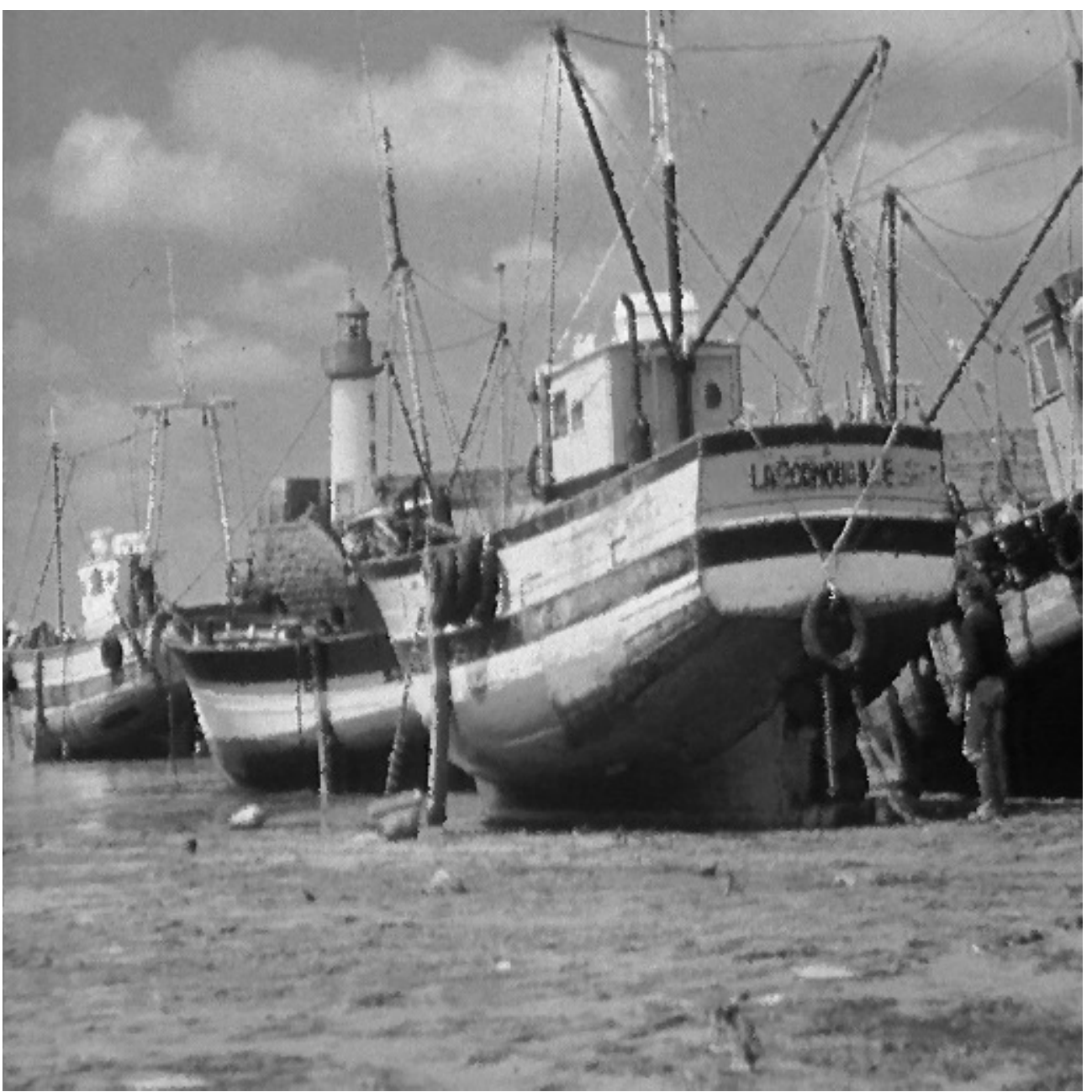}}\\
{\includegraphics[width=0.24\linewidth]{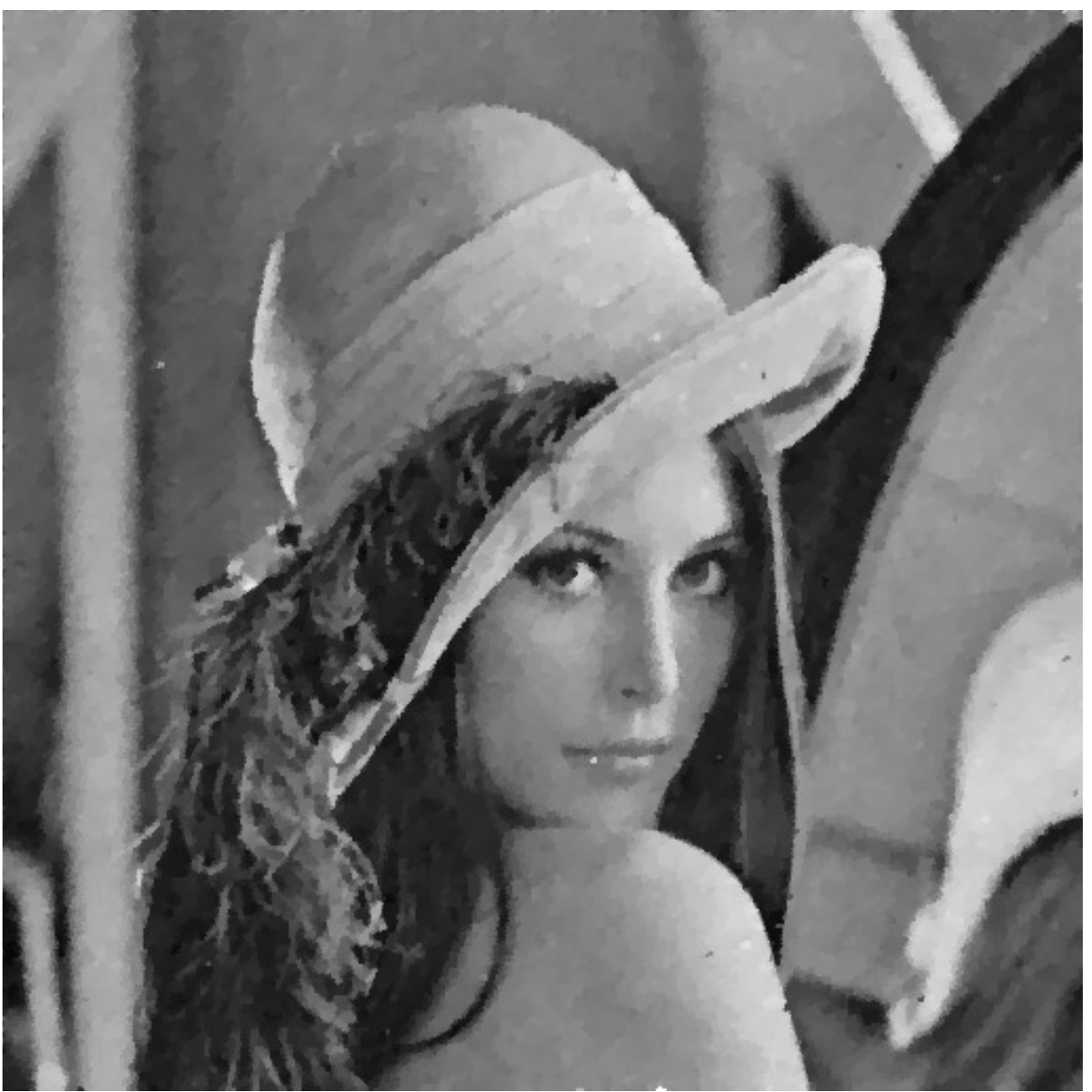}}
{\includegraphics[width=0.24\linewidth]{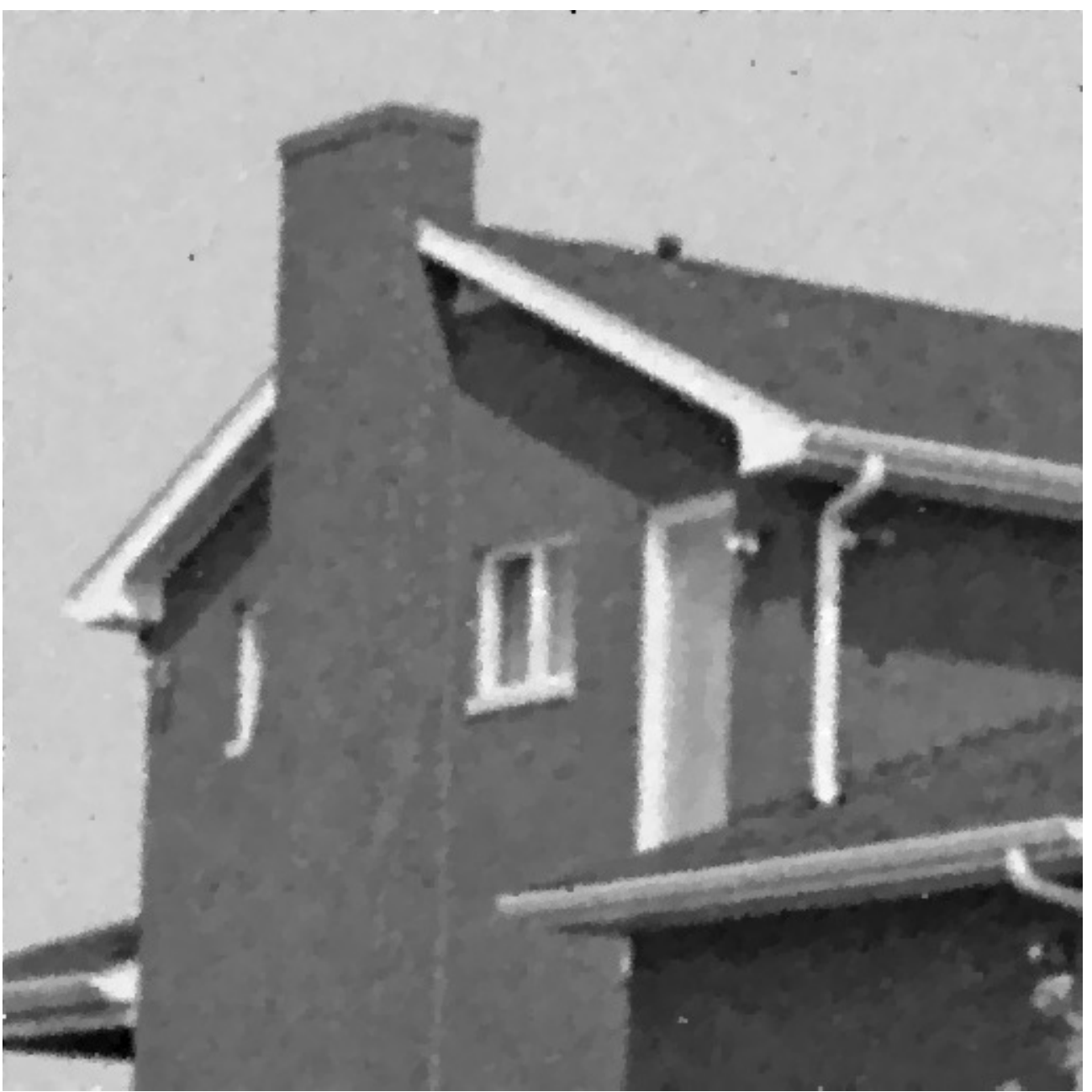}}
{\includegraphics[width=0.24\linewidth]{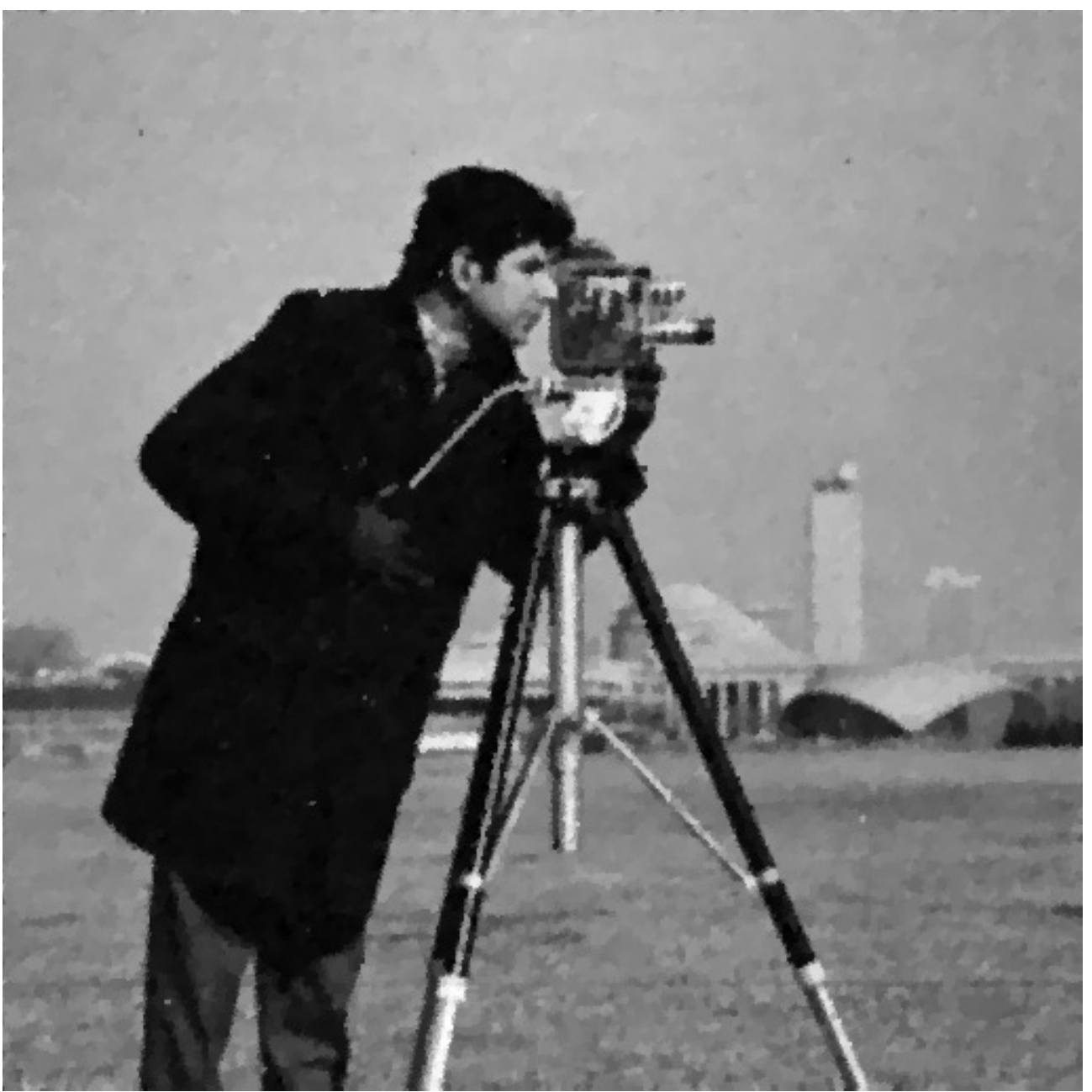}}
{\includegraphics[width=0.24\linewidth]{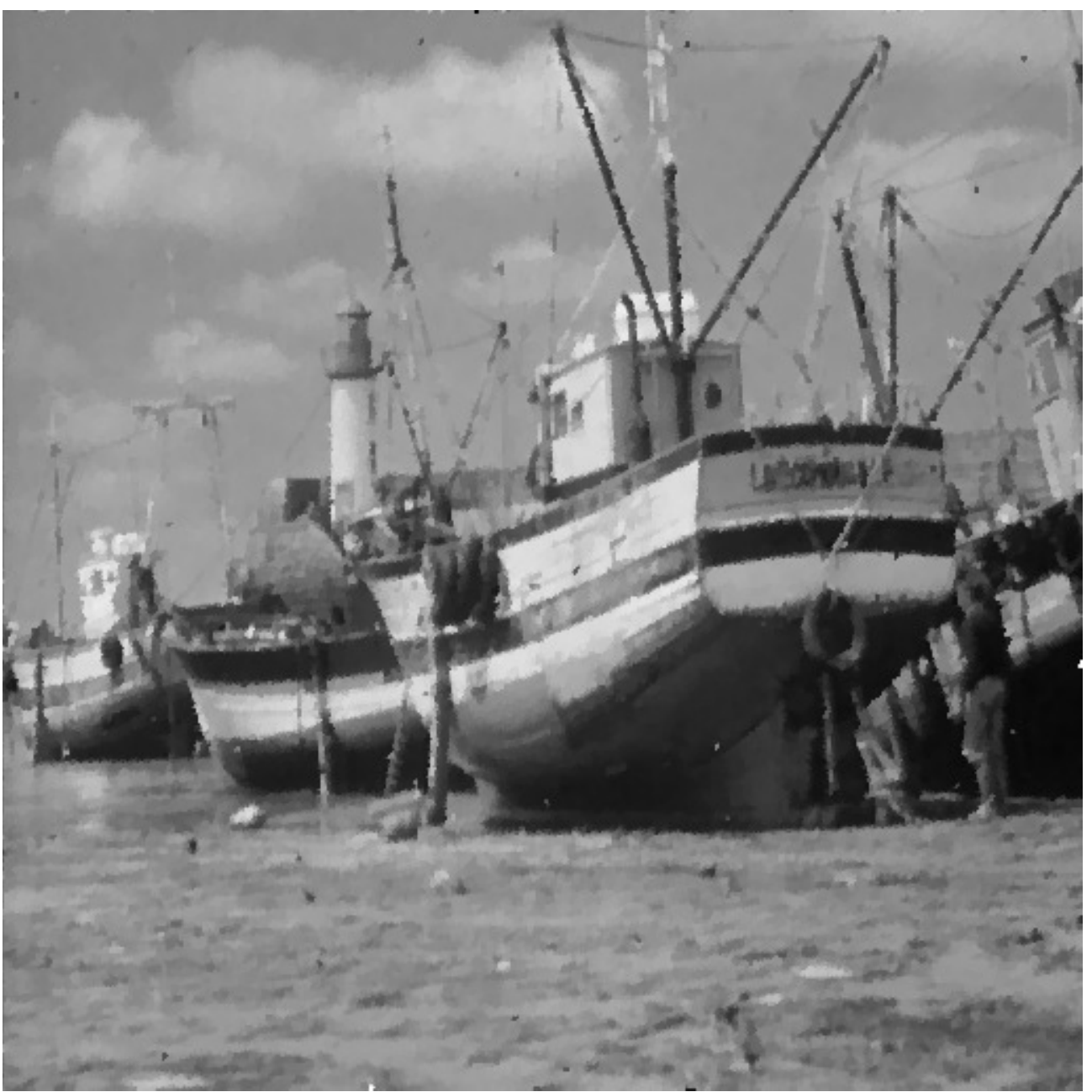}}\\
{\includegraphics[width=0.24\linewidth]{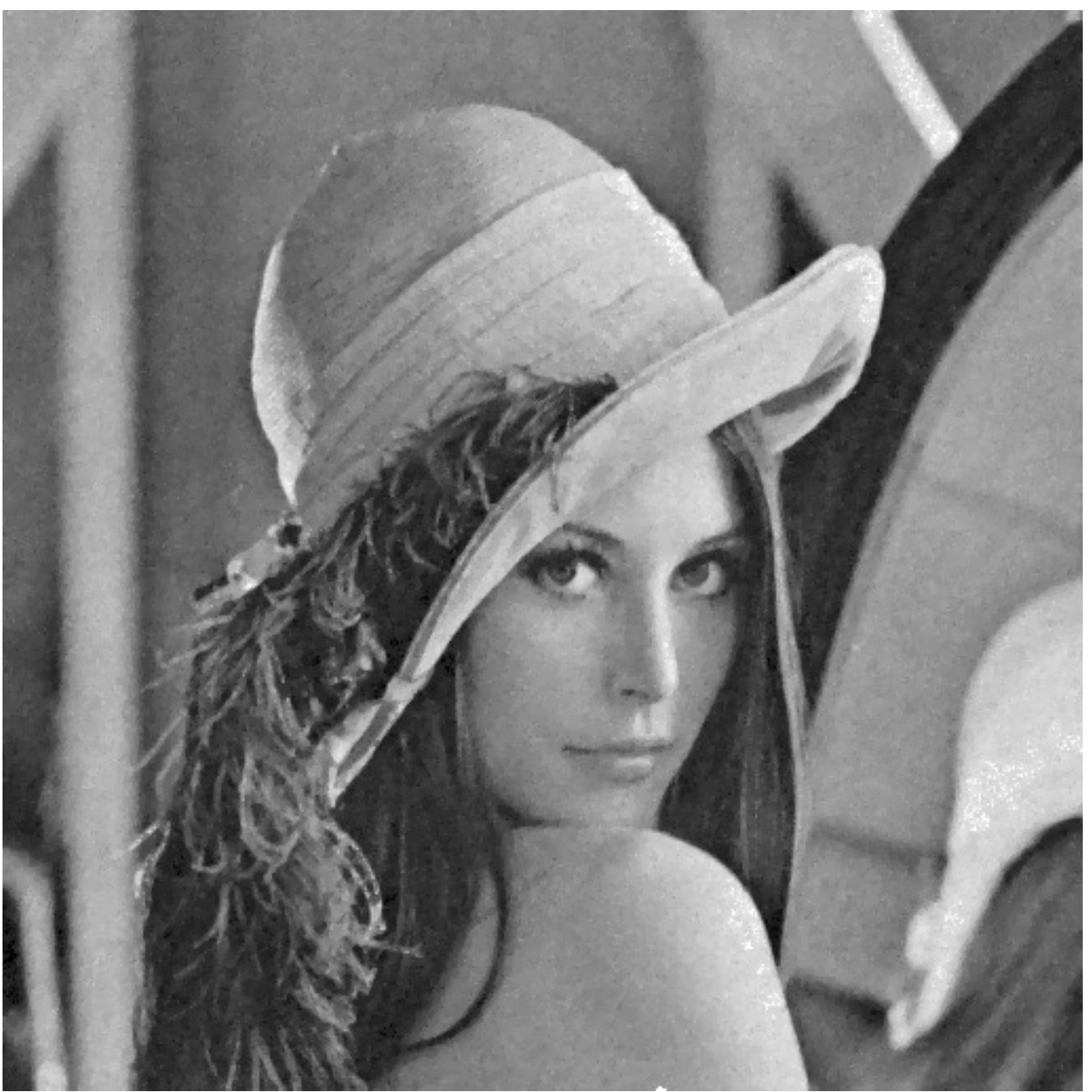}}
{\includegraphics[width=0.24\linewidth]{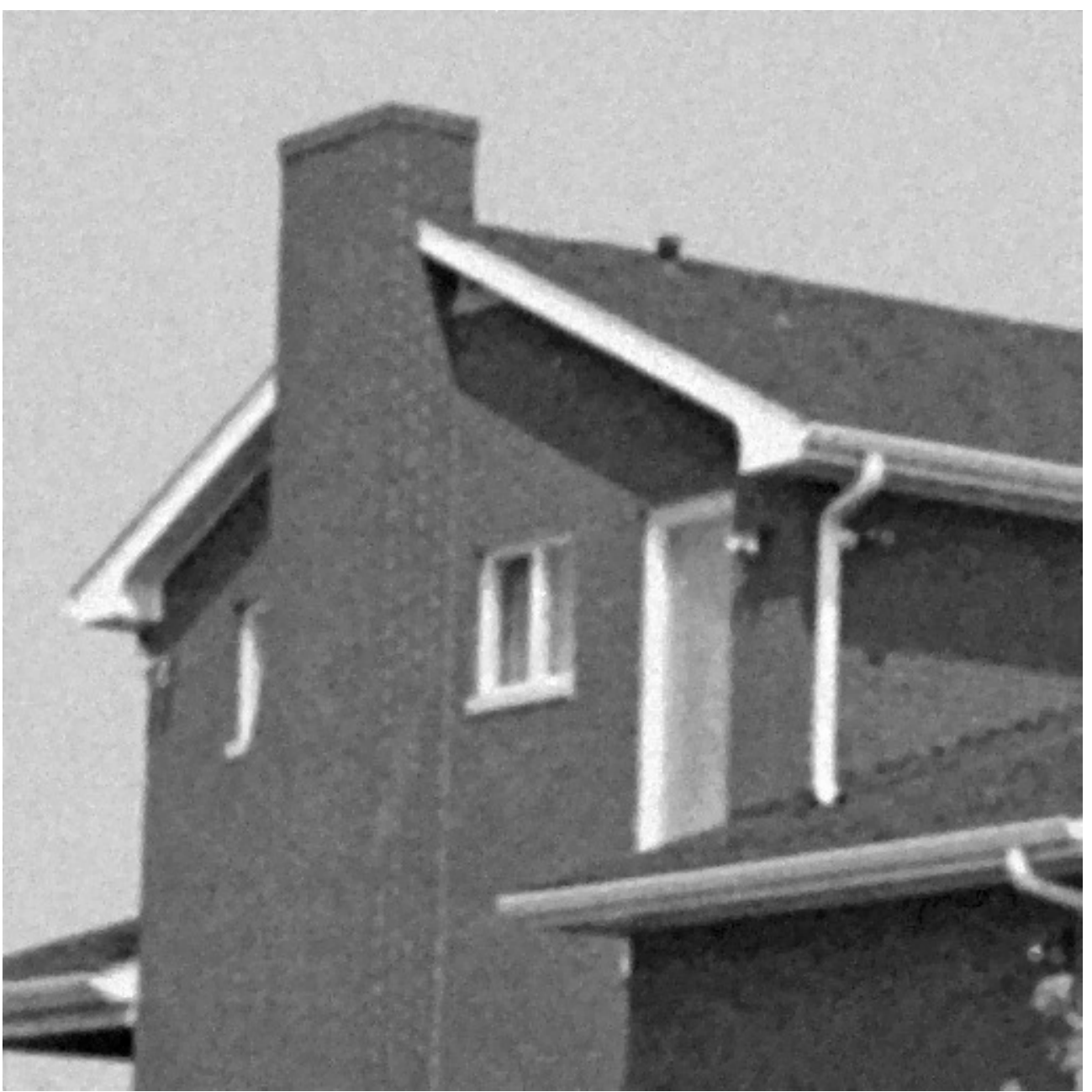}}
{\includegraphics[width=0.24\linewidth]{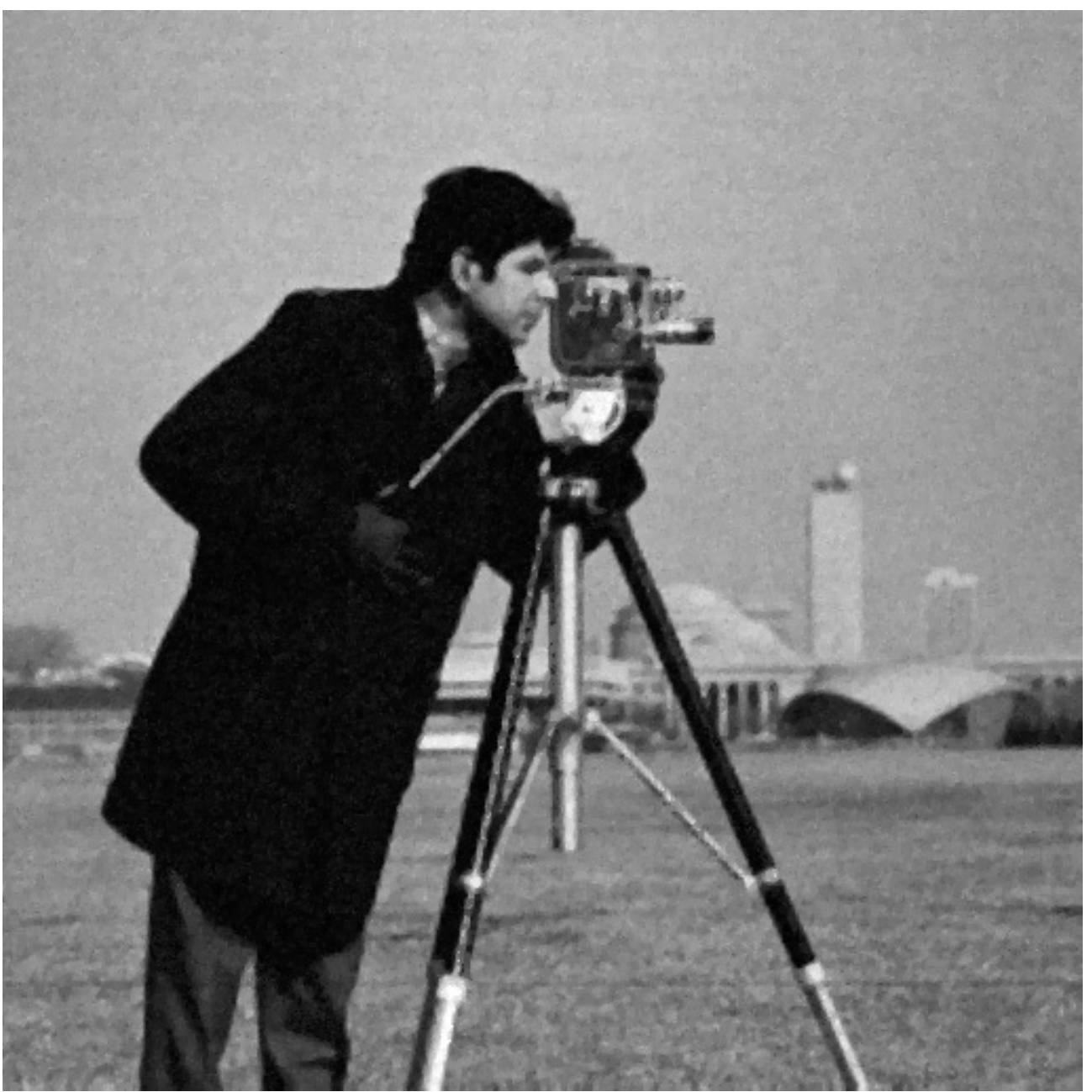}}
{\includegraphics[width=0.24\linewidth]{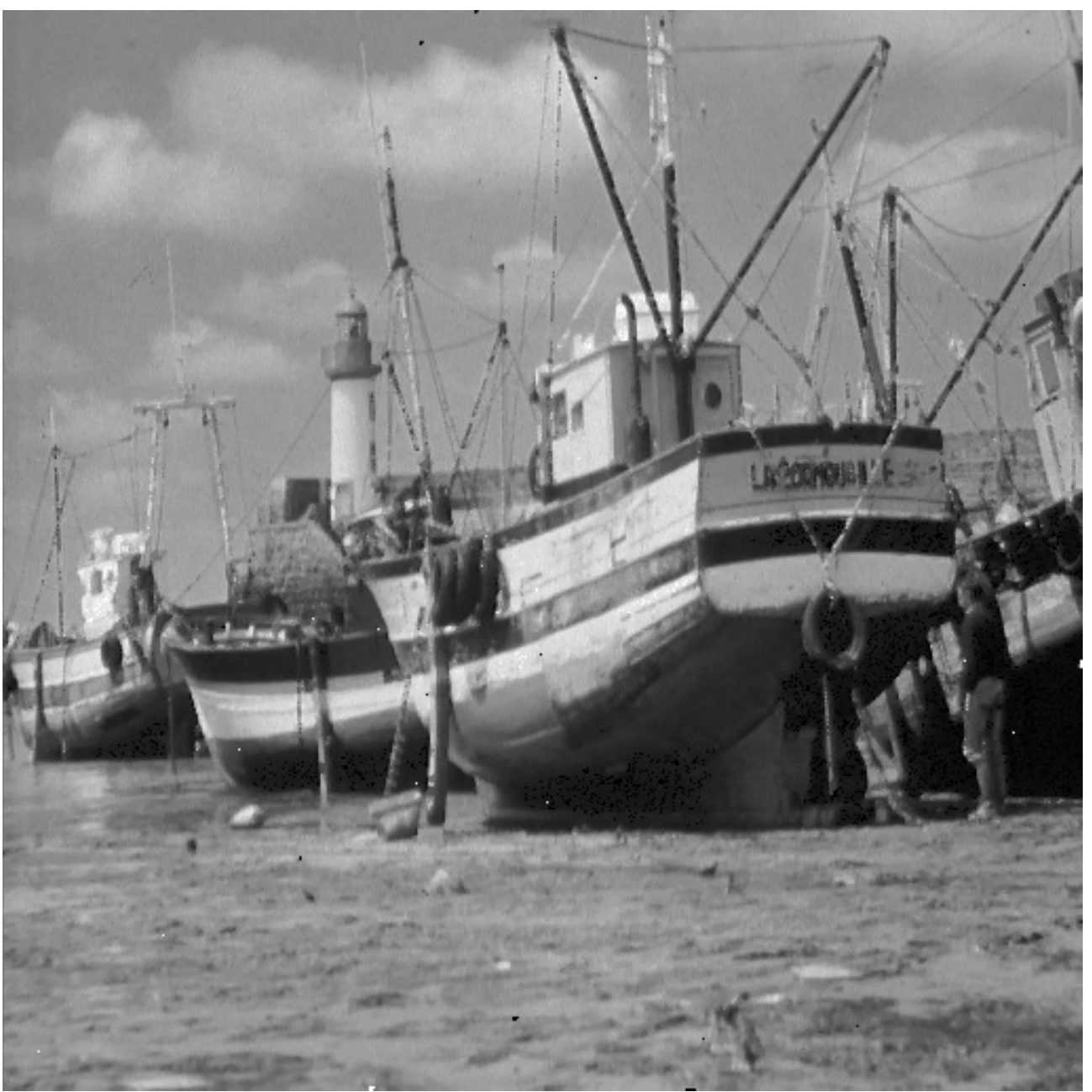}}
\caption{Denoising results of images contaminated by both Gaussian noise and salt-and-pepper impulse noise
with $\sigma=10$ and $s=30\%$. Top row: noisy images; Second row: the results restored by AMF; Third row: the results restored by TVL1; Bottom row: the results restored by total variation blind inpainting using AOP.}
\label{fig:Denoise_10_30}
\end{center}
\end{figure*}

For random-valued impulse noise removal, it is more difficult to detect the corrupted pixels because they can take any value between $d_{\small min}$ and $d_{\small max}$. ACWMF and ROLD are used in two-stage approaches for detecting the damaged pixels~\cite{ACWMF,DoCX07}. In this experiment, we will compare AOP with ACWMF, TVL1, and two-stage approaches for random-valued impulse noise removal. The two-stage approach we used here is just one step of AOP, and the parameter for second stage (total variation image inpainting) is also tuned to achieve the best quality of the restoration images. In addition, we will compare with two other methods, one is Dong et al.'s method~\cite{BlindInpainingDong} for solving problem
\begin{align}
 \Min_{u,v} {1\over2}\sum_{i,j}((Hu)_{i,j}+v_{i,j}-f_{i,j})^2+\lambda_1 TV(u) + \lambda_2\sum_{i,j}|v_{i,j}|.
\end{align}
The parameters $\lambda_1$ and $\lambda_2$ are chosen to achieve the best quality. The other is the penalty decomposition method (PD) by Lu and Zhang~\cite{PenaltyL02} for the problem
\begin{align}\left.\begin{array}{rl}
 \Min\limits_{u,v}\ &{1\over2}\sum_{i,j}((Hu)_{i,j}+v_{i,j}-f_{i,j})^2+\lambda_1 TV(u),\\
 \mbox{subject to }& \|v\|_0\leq L.\end{array}\right.
\end{align}
$\lambda_1$ is tuned to achieve the best quality of the restoration images.

Again four test images are corrupted by Gaussian noise of zero mean and standard deviation ($\sigma =  10, 25$), then we add random-valued impulse noise with different levels ($s=25\%,40\%$) onto the test images, with or without Gaussian noise. The PSNR values of the results from these six methods are summarized in Table~\ref{tab:Denoise_IN2}.

\begin{table*}[!h]
\footnotesize
\begin{center}
\begin{tabular}{|c|r|ccccccc|}\hline
\multicolumn{9}{|c|}{\bf Random-Valued Impulse Noise}\\\hline
   &$\sigma+s$ & Noisy & ACWMF & TVL1&Two-Stage &\cite{BlindInpainingDong}& PD& AOP   \\\hline
\multirow{6}{*}{``Lena"} &   0+25\%& 15.25 & 30.53& 31.75&32.65 &31.09&33.42&{\bf 33.74}\\\cline{2-9}
 &10+25\% & 15.26 & 30.44 & 31.87 &32.65 &31.21 & 33.42&{\bf 33.66} \\\cline{2-9}
 &25+25\% & 15.11 & 29.23 & 30.37 &31.32 & 30.13& 32.36&{\bf 32.64}   \\\cline{2-9}
   & 0+40\%& 13.27 & 24.62 & 29.22&28.44 &28.61&30.39&{\bf 30.77}\\\cline{2-9}
& 10+40\% & 13.22 & 24.31 & 28.94 &28.15 &28.41&30.14&{\bf 30.34}   \\\cline{2-9}
 &25+40\% & 13.05 & 23.42 & 28.02 &27.23 &27.40&29.23&{\bf 29.54} \\\hline
\multirow{6}{*}{``House"}   & 0+25\% & 14.71 & 31.50 & 36.88&35.55 &35.04&41.38& {\bf 42.11}\\\cline{2-9}
&10+25\% & 14.71 & 31.38 & 36.36 &35.55&35.20&41.36& {\bf 41.61}  \\\cline{2-9}
&25+25\% & 14.49 & 30.48 & 35.51 &34.62&34.33&40.54& {\bf 40.85}\\\cline{2-9}
&0+40\% & 12.65 & 23.90 & 32.77 &30.07&31.45&35.53& {\bf 37.39}\\\cline{2-9}
 &10+40\%& 12.62 & 23.82 & 32.44 &29.97&31.24&35.24& {\bf 36.71}\\\cline{2-9}
 &25+40\% & 12.34 & 22.93 & 31.52 & 29.02&30.38&34.32&{\bf 35.86} \\\hline
  \multirow{6}{*}{``Cameraman"}   &  0+25\%& 14.48 & 28.93 &31.24 &31.75&30.54&32.72& {\bf 33.16}\\\cline{2-9}
 &10+25\% & 14.52 & 28.99 & 31.32 &31.86&30.54&32.72&{\bf 33.26}  \\\cline{2-9}
 &25+25\% & 14.10 & 27.98 & 30.43 &30.93&29.63&31.80&{\bf 32.55}  \\\cline{2-9}
    & 0+40\%& 12.37 & 22.26 &27.36 & 26.51&27.00&28.39&{\bf 29.16}\\\cline{2-9}
 &10+40\% & 12.39 & 22.42 & 27.75 &26.74&27.37&28.42&{\bf 29.21} \\\cline{2-9}
 &25+40\% & 12.02 & 21.32 & 26.69 &25.56&26.30&27.43& {\bf 28.11}\\\hline
\multirow{6}{*}{``Boat"}  &0+25\%&15.29 & 28.18&28.63&29.37&28.26&29.48 & {\bf 29.60}\\\cline{2-9}
 &10+15\%& 15.35 & 28.16 & 28.76 &29.48&28.36&29.61& {\bf 29.78}\\\cline{2-9}
 &25+25\%& 14.01 & 27.03 & 27.73 &28.58&27.43&28.62& {\bf 28.81}\\\cline{2-9}
 &0+40\%&13.30 &23.56 & 26.16& 26.03&25.88&26.88&{\bf 27.12}\\\cline{2-9}
  &10+40\%& 13.31 & 23.42 & 26.22 & 25.96&25.90&26.88&{\bf 26.99}\\\cline{2-9}
   &25+40\%& 13.12 & 22.37 & 25.27 &25.00 &24.79&25.93&{\bf 26.02}\\\hline
\end{tabular}
\caption{PSNR(dB) for denoising results of different algorithms for noisy images corrupted by random-valued impulse noise and mixed Gaussian  impulse noise. $\sigma$ is the standard deviation for the Gaussian noise and $s$ is the level of random-valued impulse noise.}
\label{tab:Denoise_IN2}
\end{center}
\end{table*}

From Table \ref{tab:Denoise_IN2}, we can see that for random-valued impulse noise, the results from total variation blind inpainting using AOP are better than those by other methods for all noise levels. The comparison of ACWMF and TVL1 shows that TVL1 outperforms ACWMF for all noise levels tested, because ACWMF misses quite a lot of real noise and false-hits some noise-free pixels. TVL1 has better performance than two-stage approach for the cases when noise level is high ($s=40\%$ in the numerical experiments), because the accuracy of detecting corrupted pixels by random-valued impulse noise using ACWMF is  very low when the noise level is high. The accuracy of detecting corrupted pixels can be improved by our method via iteratively updating the binary matrix $\Lambda$, as shown in the comparison. PD outperforms other methods except AOP because $\ell_0$ term is used in the problem. The problem PD solves is a non-continuous problem and it will stop at a local minimum $(u^*,v^*)$. However, $u^*$ may not be a local minimum of the problem with $u$ only. While AOP will converge to a local minimum of the problem with $u$ only.

The visual comparison of some results is shown in Fig.~\ref{fig:Denoise_10_25}. We can see noisy artifacts in the background of the images obtained by ACWMF, and the images obtained by TVL1 are blurred with some lost details. Images restored by total variation blind inpainting are smooth in flat regions of the background and the details are kept.

\begin{figure*}[!h]
\begin{center}
{\includegraphics[width=0.24\linewidth]{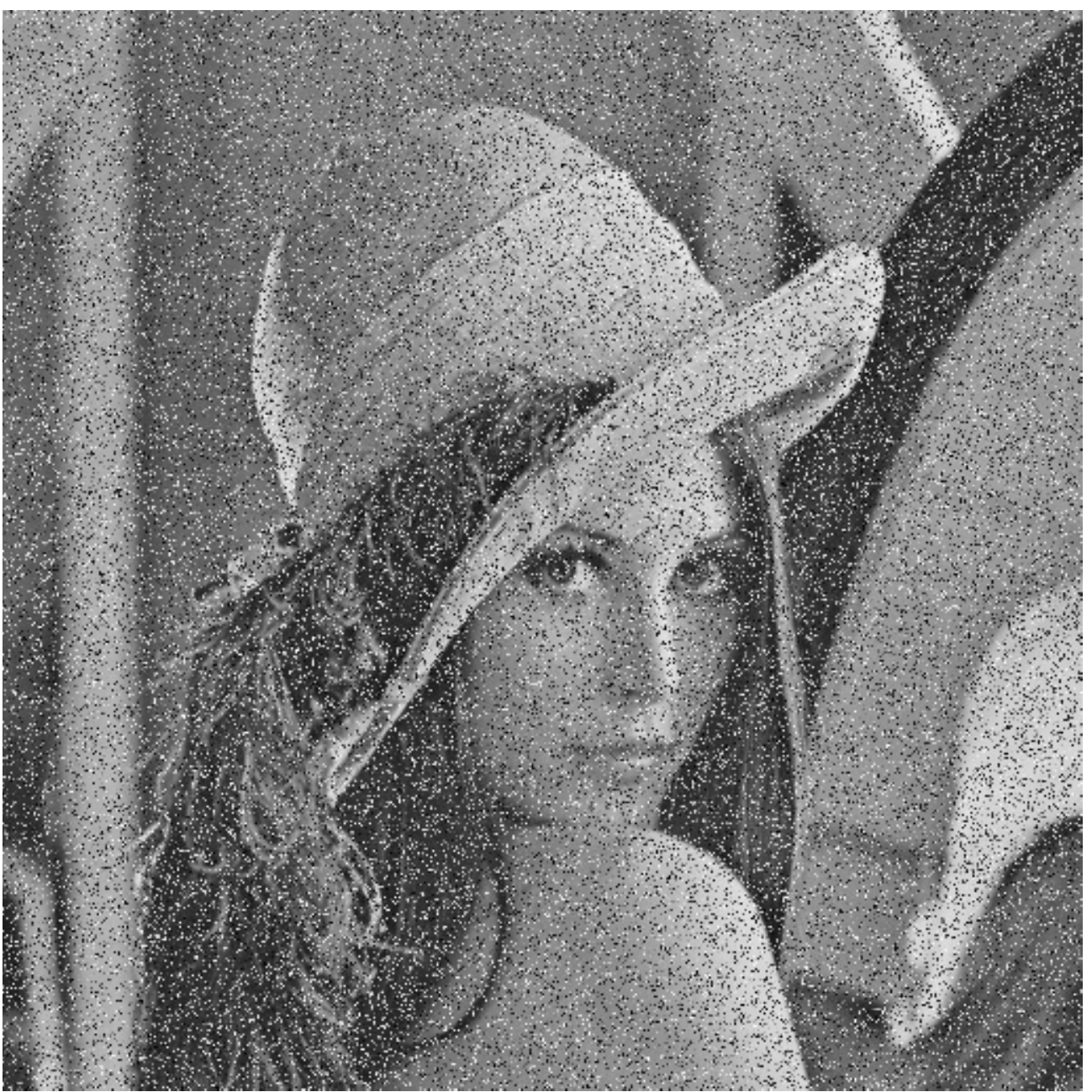}}
{\includegraphics[width=0.24\linewidth]{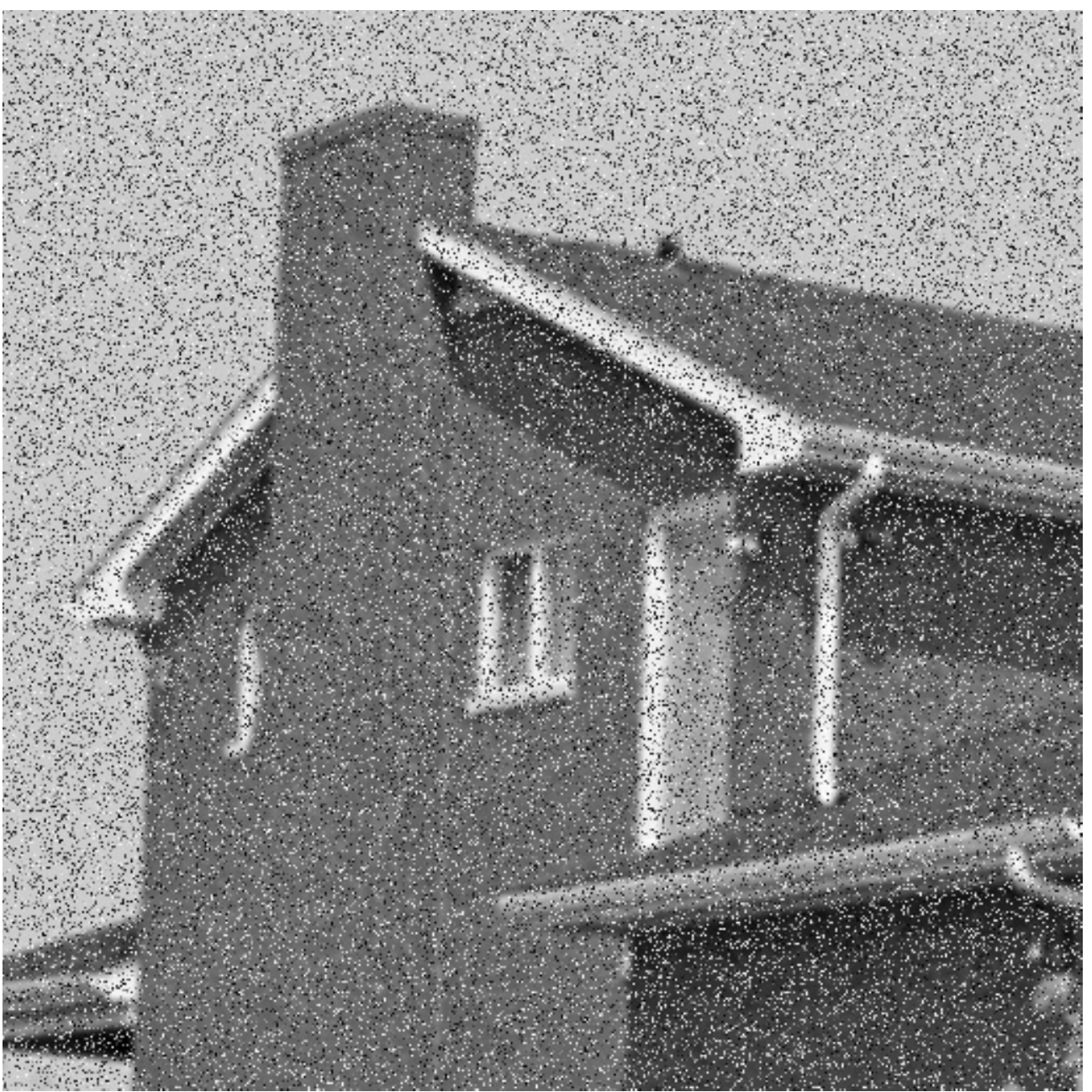}}
{\includegraphics[width=0.24\linewidth]{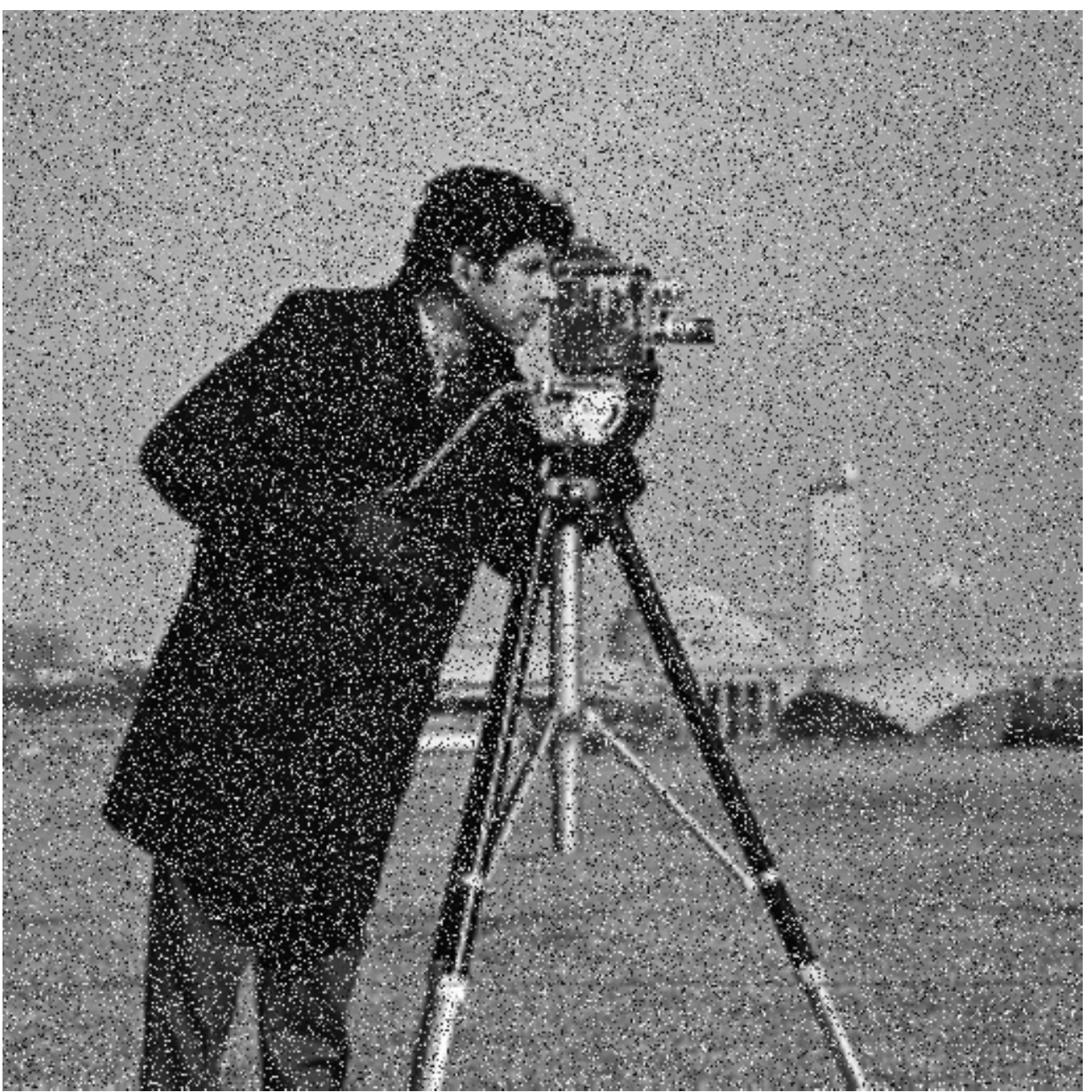}}
{\includegraphics[width=0.24\linewidth]{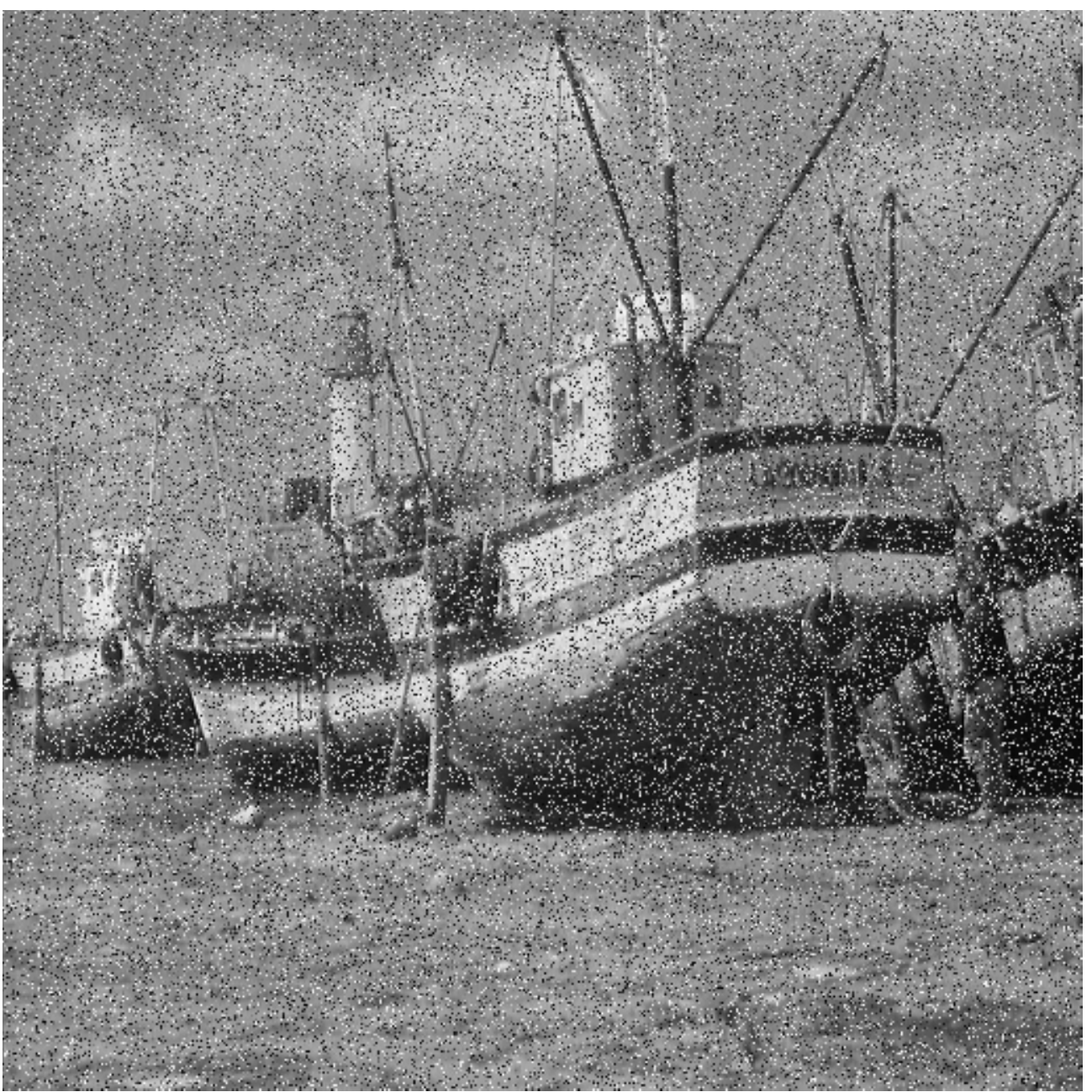}}\\
{\includegraphics[width=0.24\linewidth]{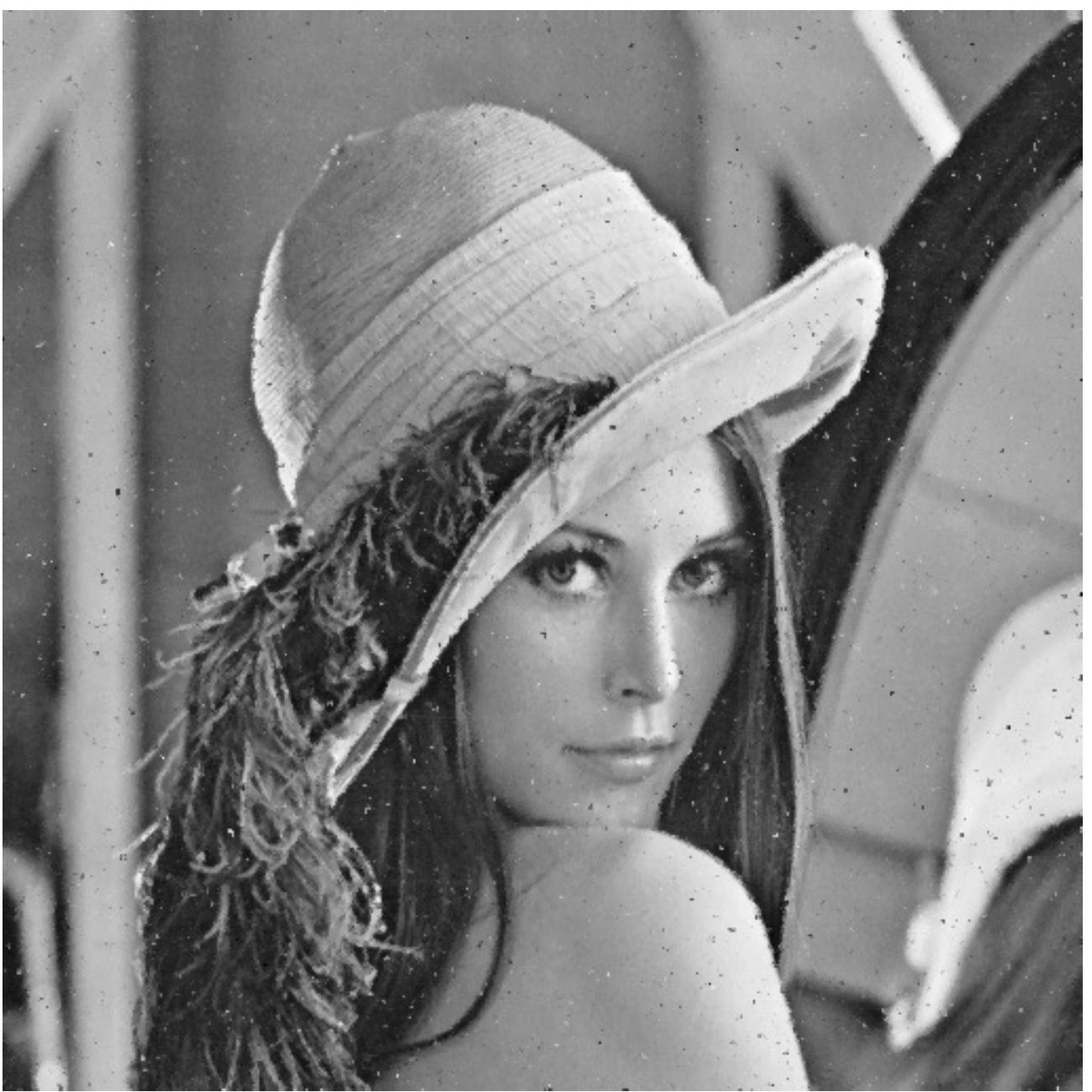}}
{\includegraphics[width=0.24\linewidth]{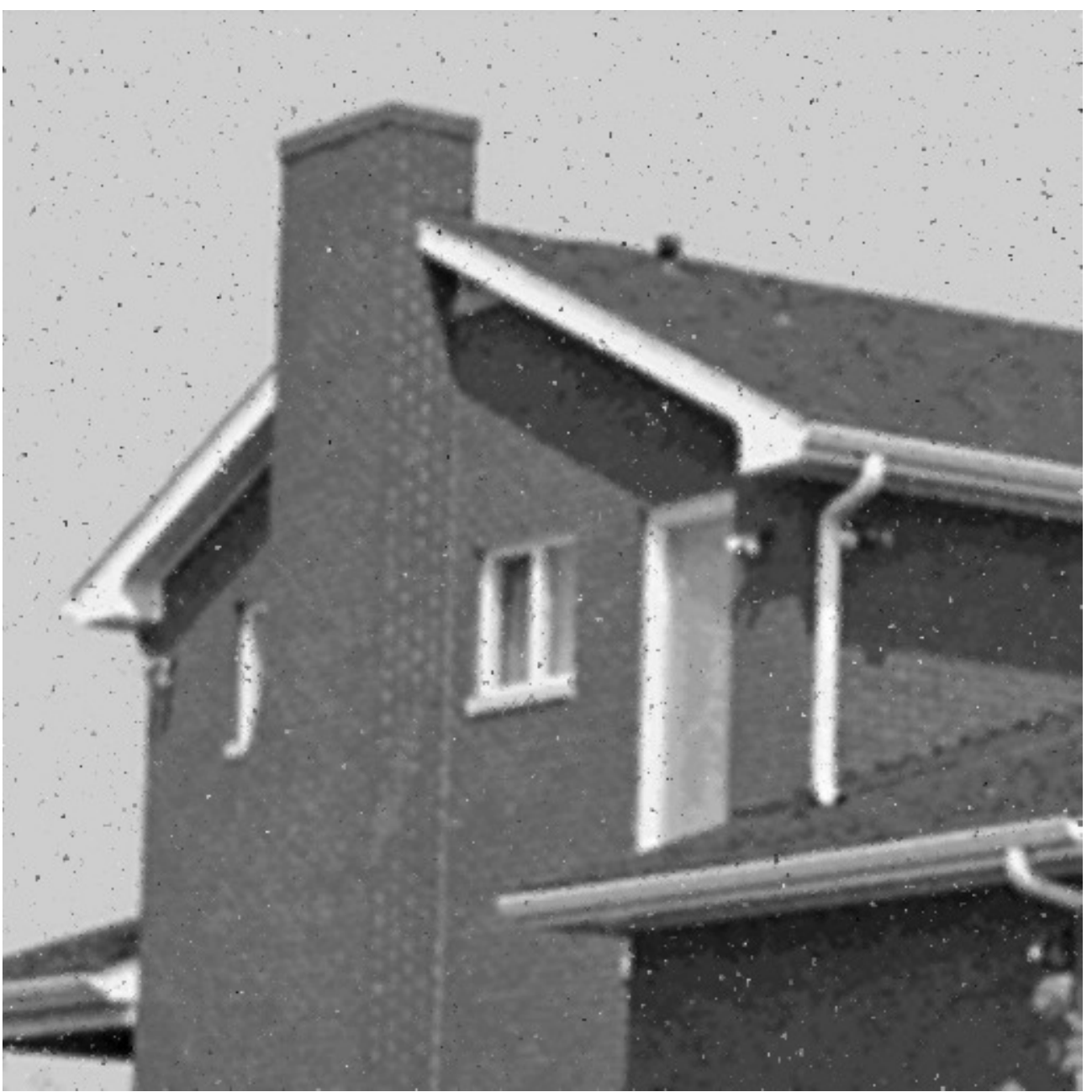}}
{\includegraphics[width=0.24\linewidth]{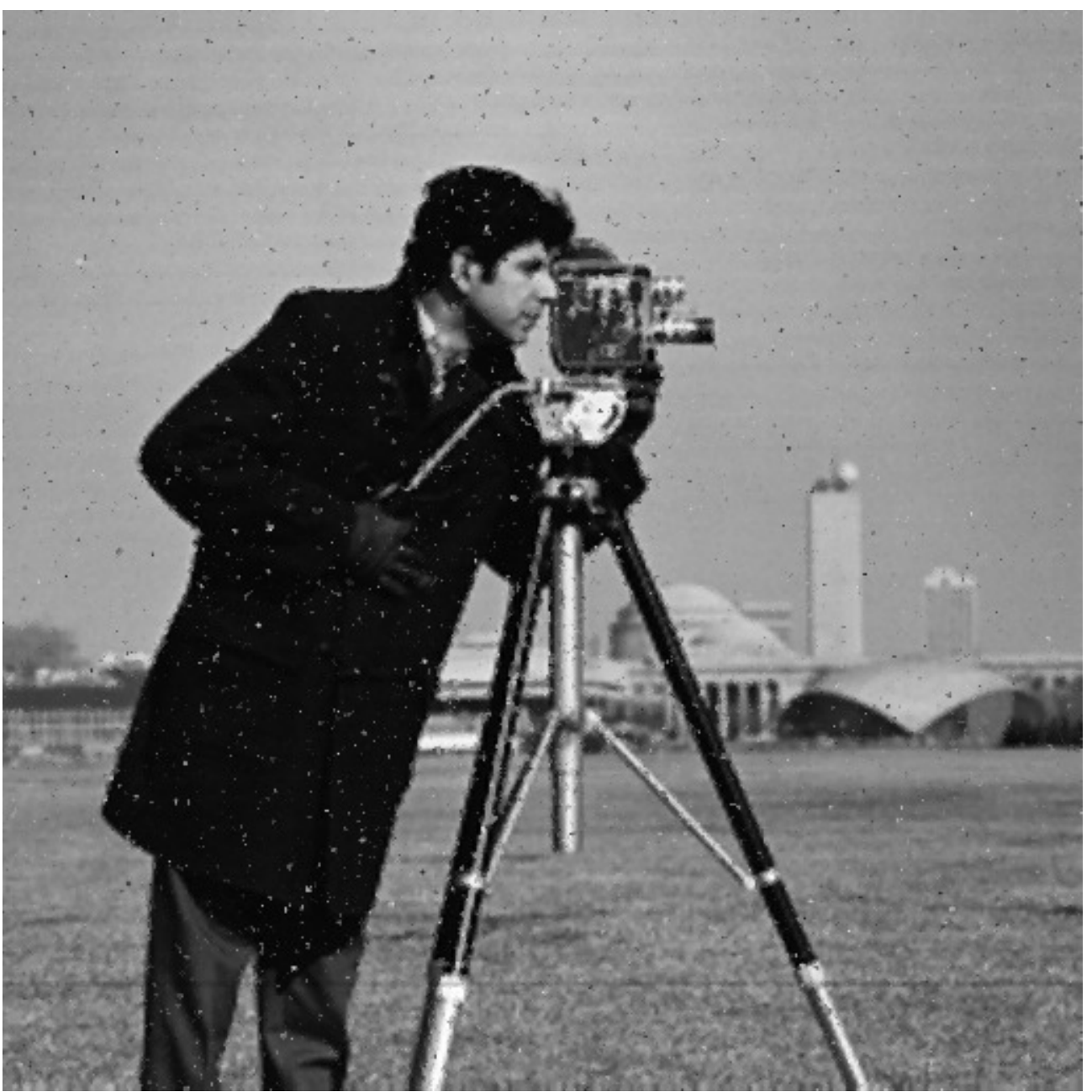}}
{\includegraphics[width=0.24\linewidth]{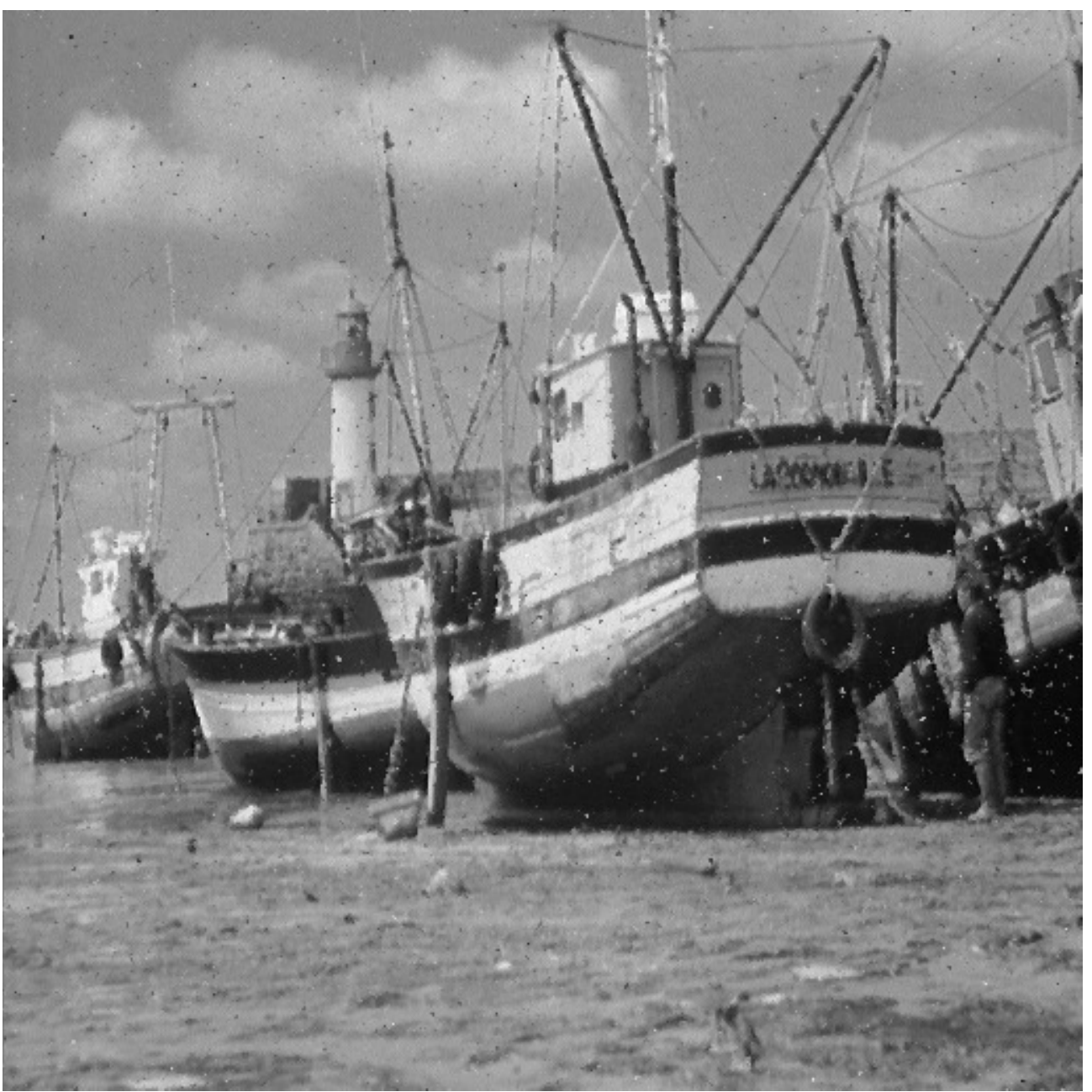}}\\
{\includegraphics[width=0.24\linewidth]{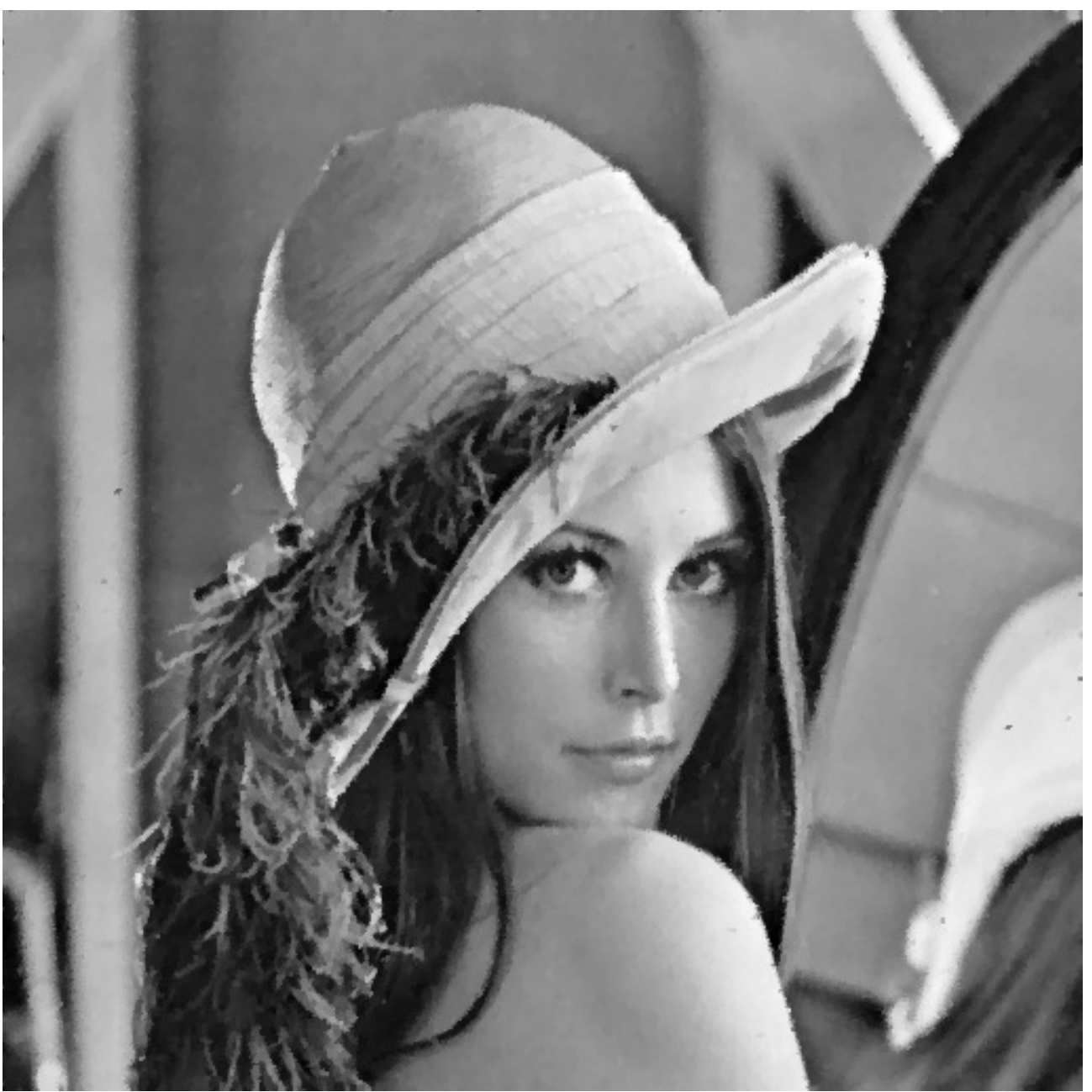}}
{\includegraphics[width=0.24\linewidth]{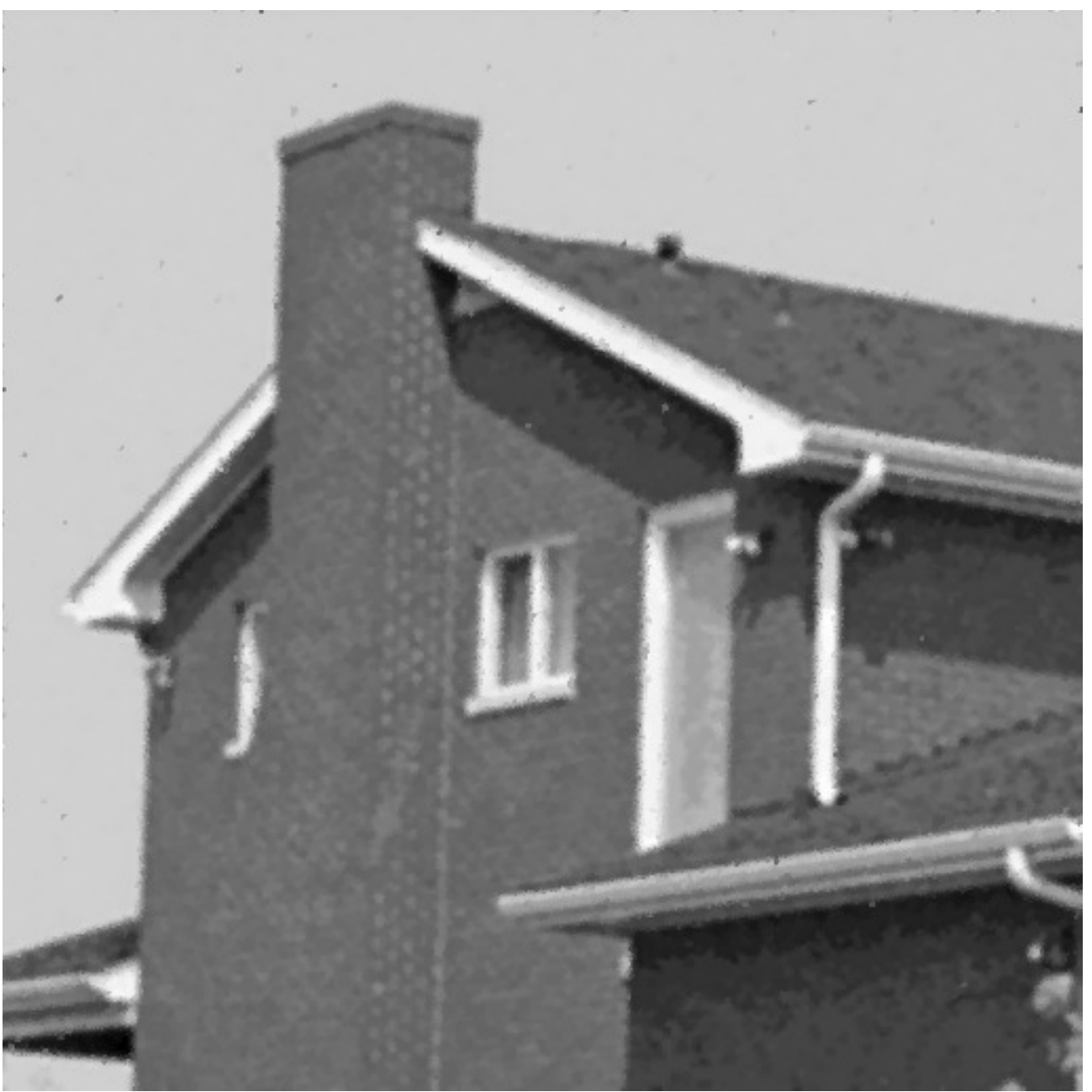}}
{\includegraphics[width=0.24\linewidth]{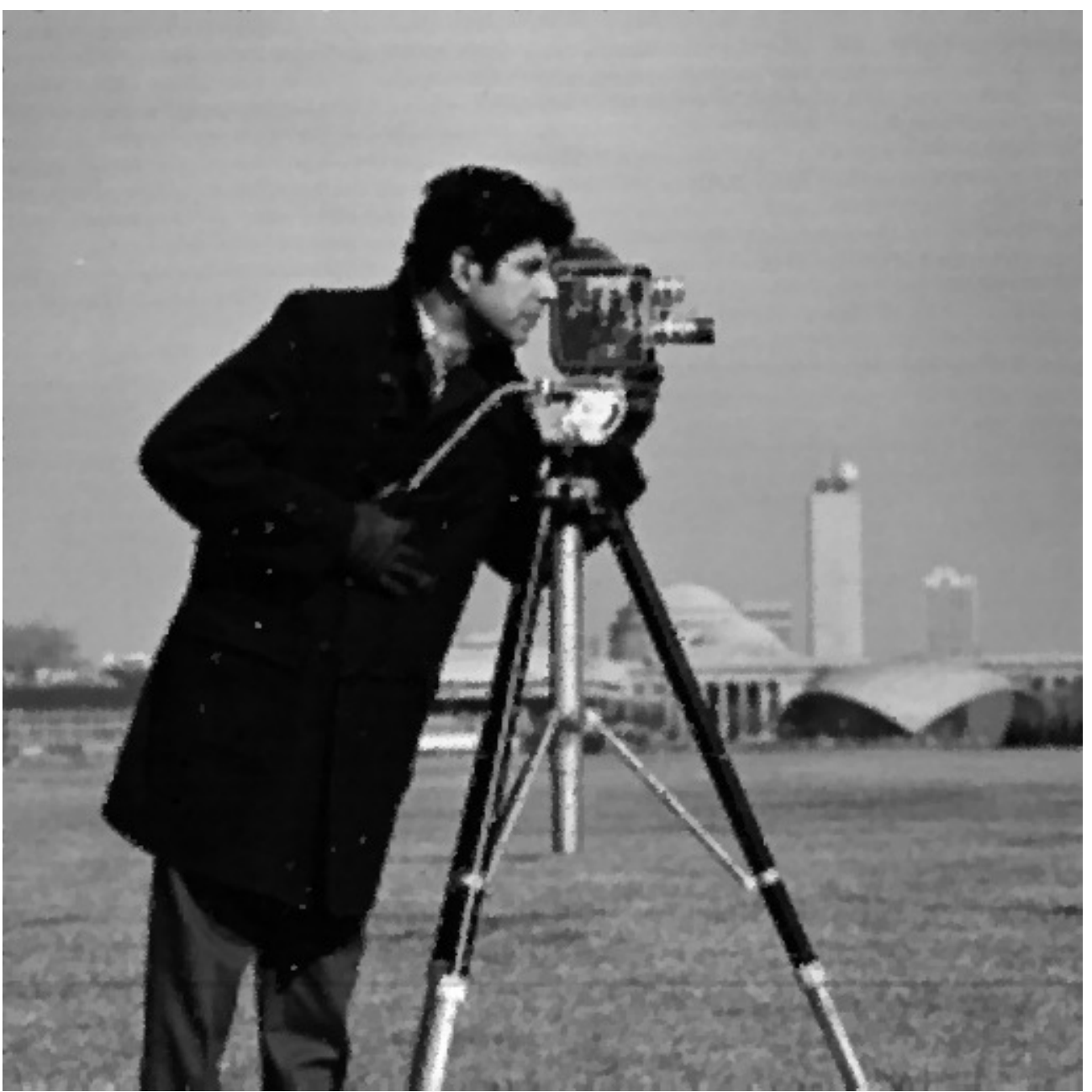}}
{\includegraphics[width=0.24\linewidth]{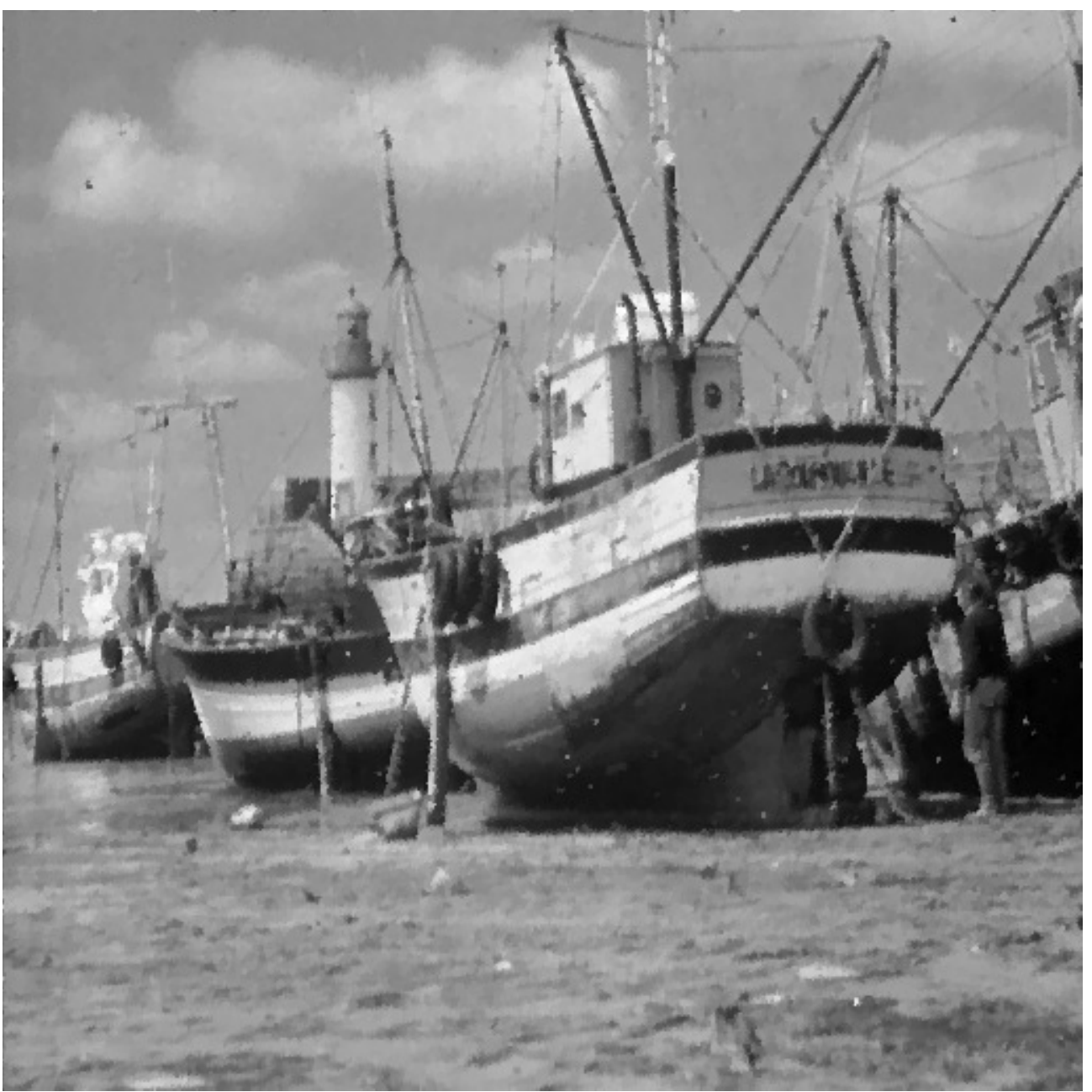}}\\
{\includegraphics[width=0.24\linewidth]{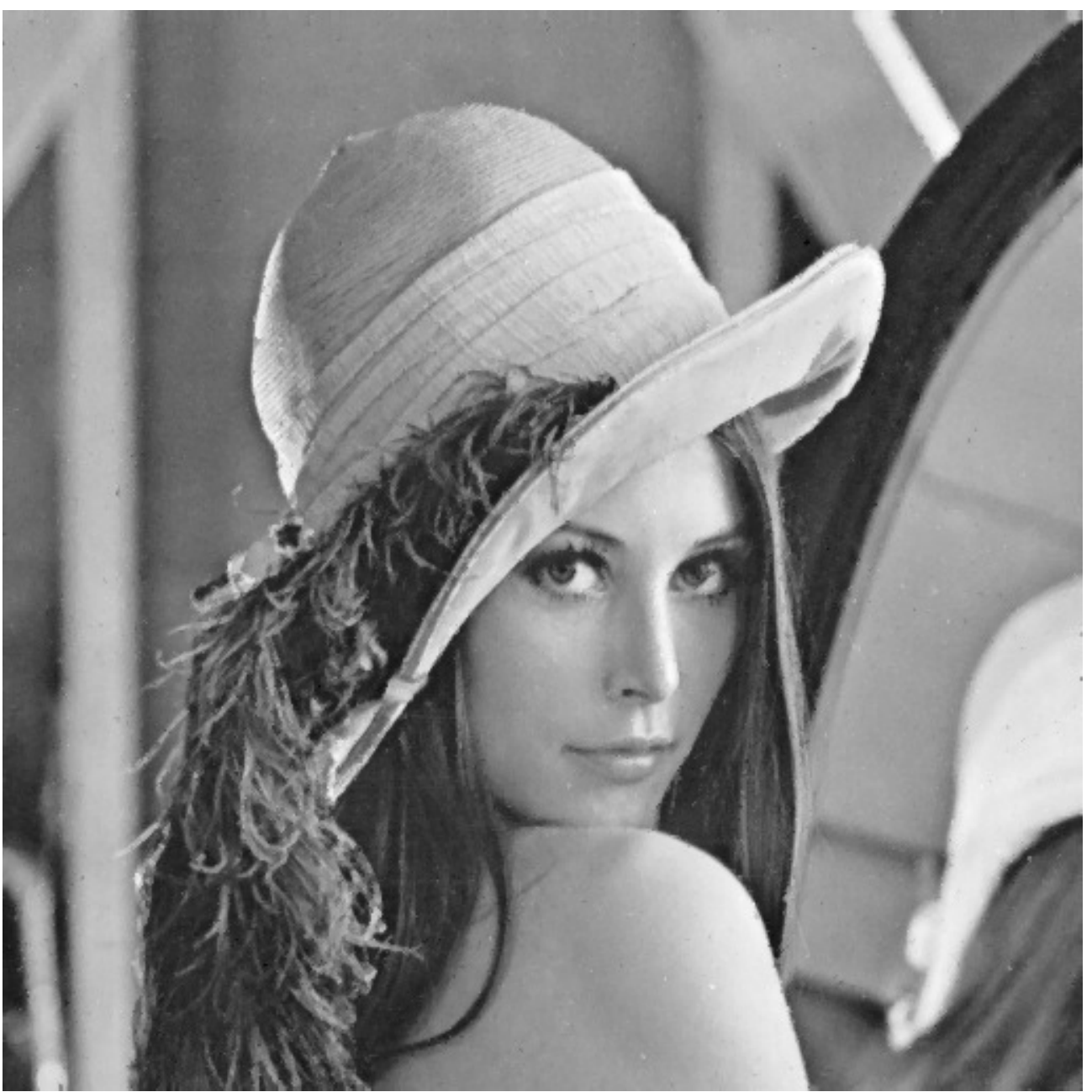}}
{\includegraphics[width=0.24\linewidth]{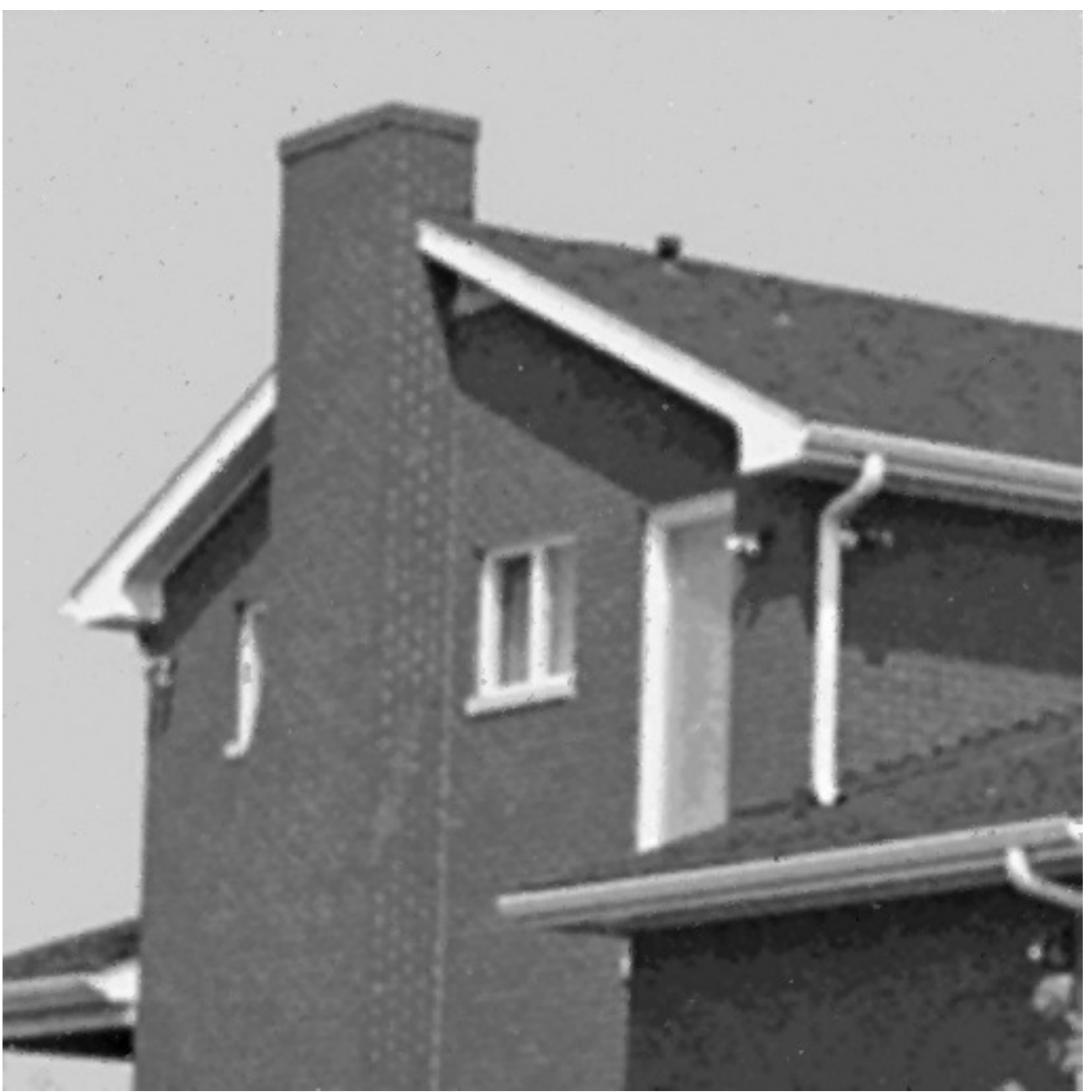}}
{\includegraphics[width=0.24\linewidth]{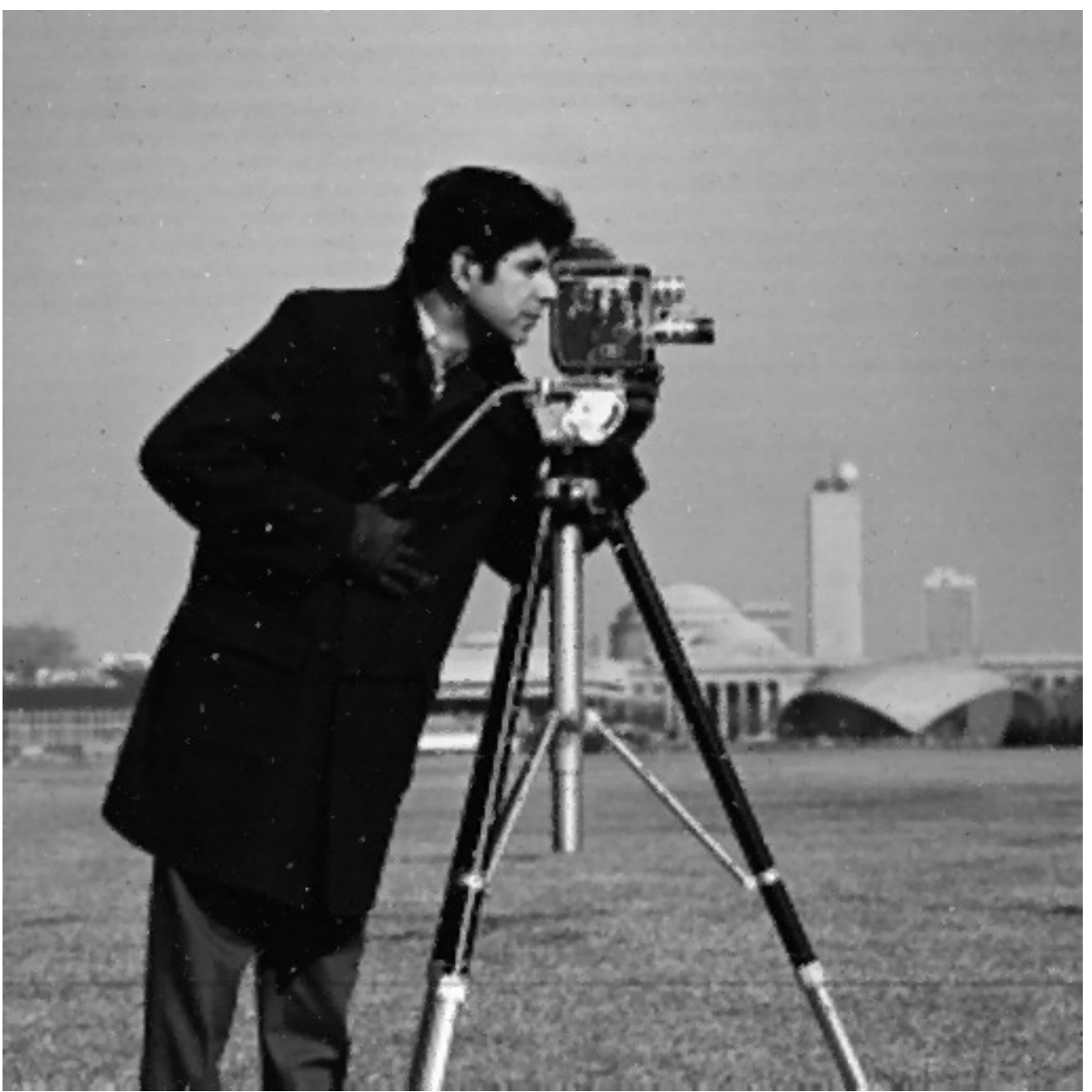}}
{\includegraphics[width=0.24\linewidth]{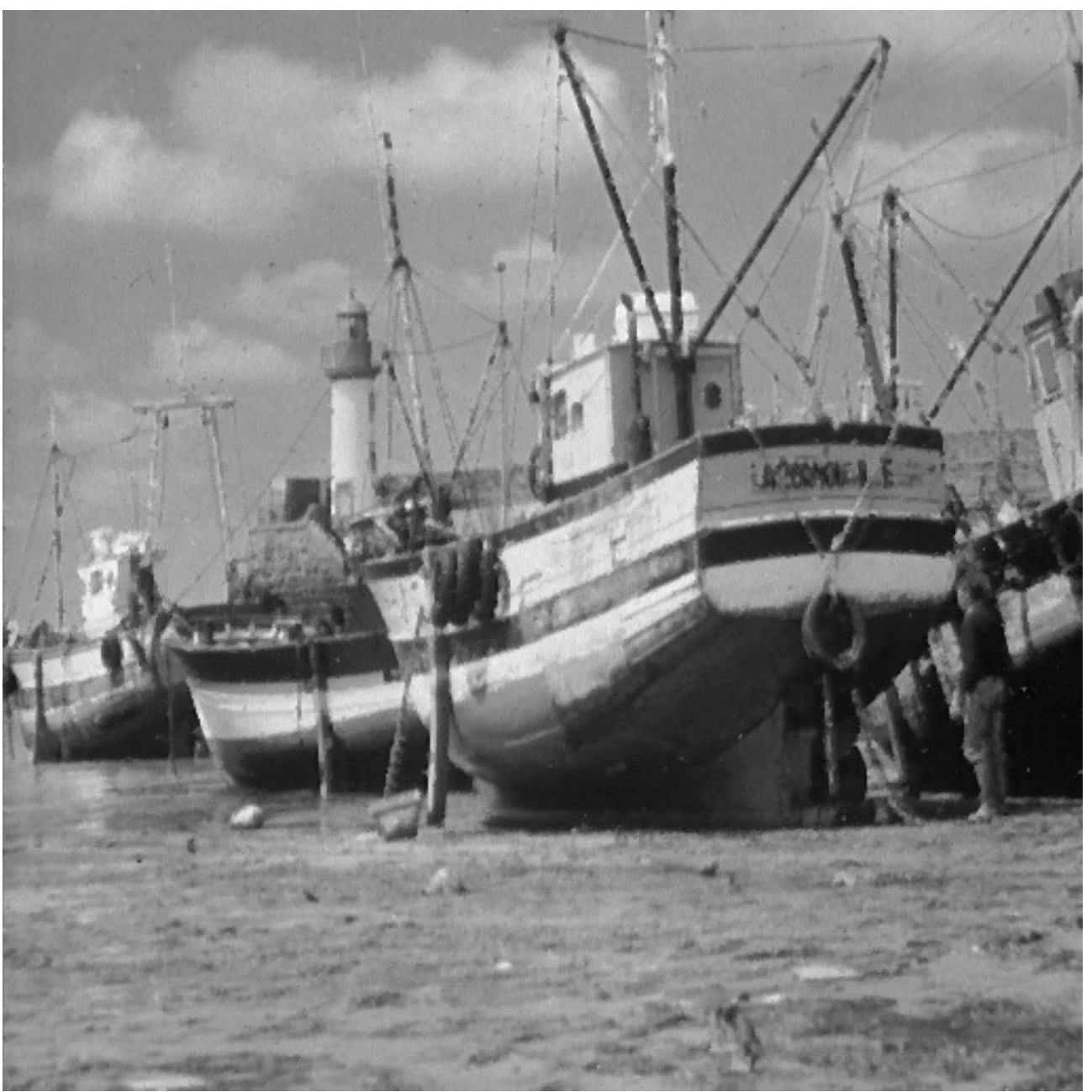}}
\caption{Denoising results of images contaminated by both Gaussian noise and random-valued impulse noise
with $\sigma=10$ and $s=25\%$. Top row: noisy images; Second row: the results restored by ACWMF; Third row: the results restored by TVL1; Bottom row: the results restored by total variation blind inpainting using AOP.}
\label{fig:Denoise_10_25}
\end{center}
\end{figure*}

Both experiments show that our method by iteratively updating the inpainting region and performing image inpainting provides better results in identifying the outliers and recovering damaged pixels. For salt-and-pepper impulse noise, because there are very accurate methods for detecting the corrupted pixels such as AMF, our method has similar performance as two-stage approaches. However, for random-valued impulse noise, there is no method can detect corrupted pixels accurately, especially when the noise level is high. Our method by iteratively updating the corrupted pixels is a better choice.

At the end of this section, we compare the damaged pixels detected by ACWMF and obtained from AOP in Fig.~\ref{fig:DetectionCompare} for the cameraman image. The damaged pixels are chosen randomly  ($s=40\%$). The pixels with black color are detected as damaged. The set obtained from AOP is also random and does not contain any information from the image, while the set detected by ACWMF still has some features from cameraman image.
\begin{figure}
\begin{center}
{\includegraphics[width=0.48\linewidth]{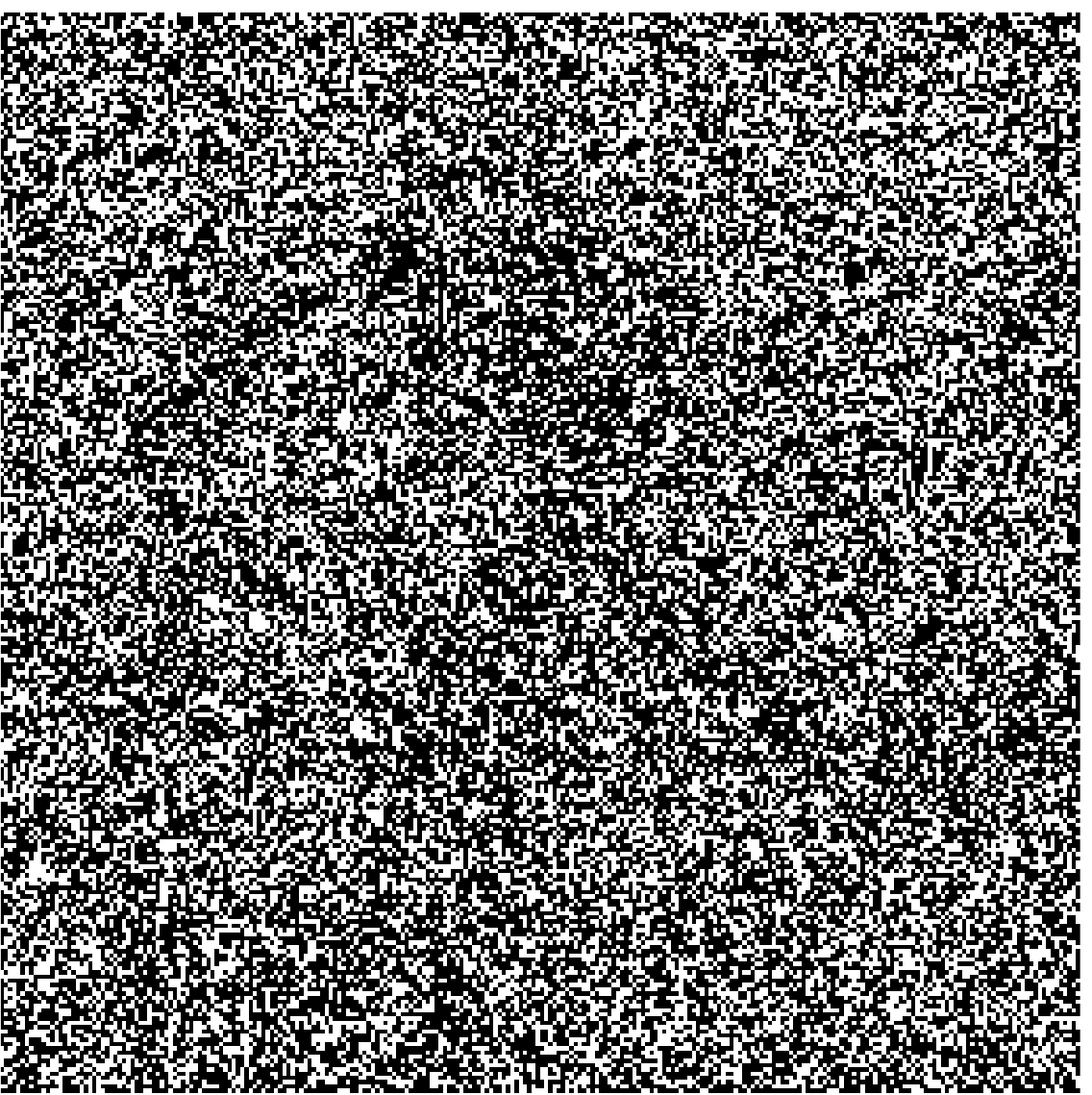}}
{\includegraphics[width=0.48\linewidth]{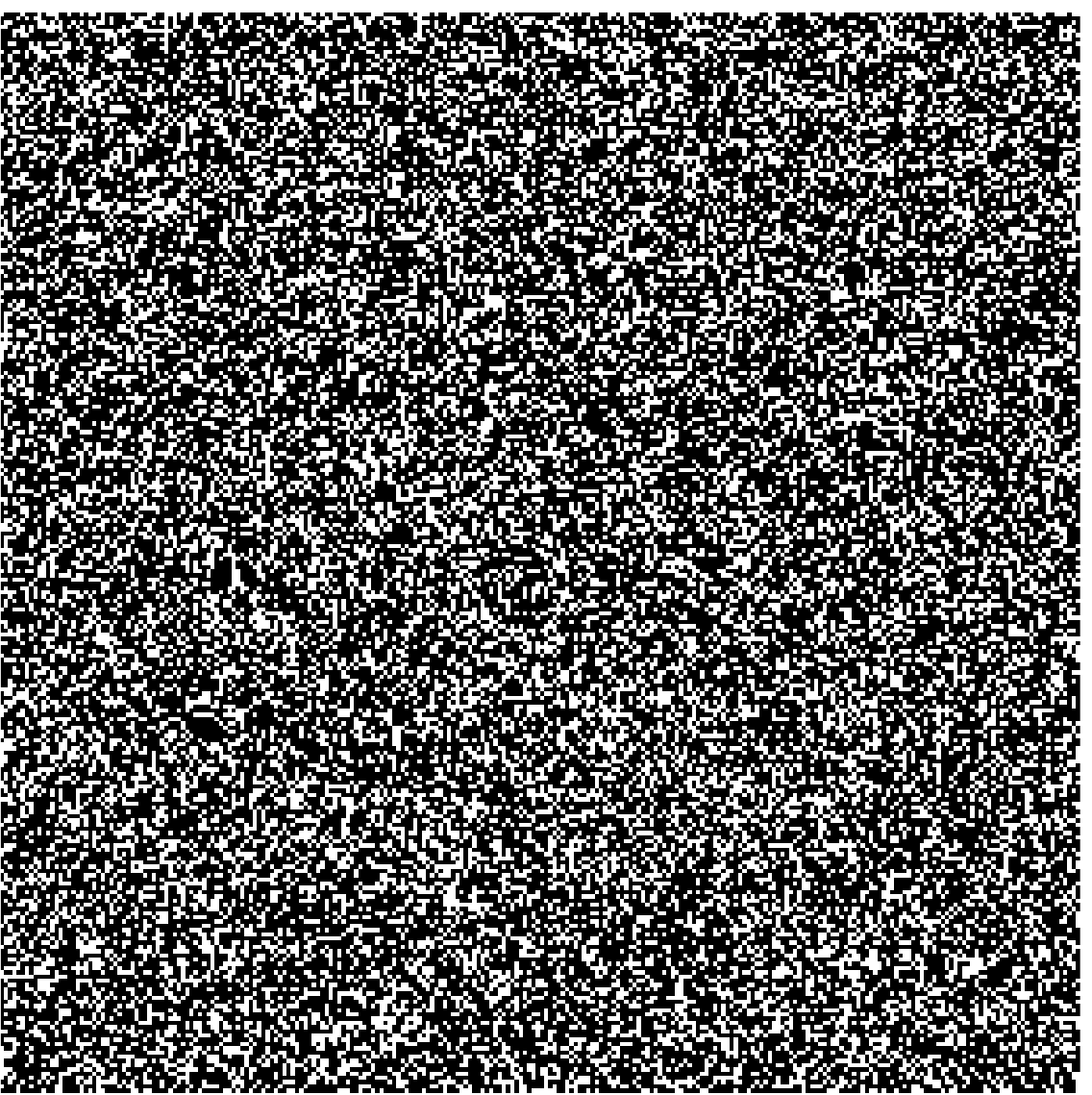}}
\caption{The damaged pixels detected by ACWMF (left column) and AOP (right column).}
\label{fig:DetectionCompare}
\end{center}
\end{figure}

\fi

\section{Conclusion}

This paper presents two general algorithms based on blind inpainting and $\ell_0$ minimization for removing impulse noise. The difference is in the treatment for the $\ell_0$ term: I) the $\ell_0$ term is put in the objective function, II) the $\ell_0$ term is in the constraint. Both problems can be solved by iteratively restoring the images and identifying the damaged pixels. The performance of these two methods is similar, and the connection between these two methods is shown. It is also shown in the experiments that the proposed methods perform better than other methods. This simple idea can also be applied to other cases where the noise model is not Gaussian.

\section*{Acknowledgment}
This work was supported by NSF Grant DMS-0714945 and Center for Domain-Specific Computing (CDSC) under the NSF Expeditions in Computing Award CCF-0926127. We would like to thank the anonymous referees for making several very helpful suggestions.

{
\bibliographystyle{IEEEtran}
\bibliography{IEEEabrv,egbib}
}

\end{document}